\documentclass{article}
\usepackage[utf8]{inputenc}
\usepackage{amsmath,amssymb,amsthm}
\usepackage{hyperref, url}
\usepackage{graphicx}
\usepackage{float}

\title{Reeb stability and the Gromov-Hausdorff limits of leaves in compact foliations}
\author{Pablo Lessa\thanks{LPMA, Université Pierre et Marie Curie, 4 place Jussieu, 75005 Paris, France.}\thanks{CMAT, Facultad de Ciencias, UdelaR, Iguá 4225, 11400 Montevideo, Uruguay. Email: lessa@cmat.edu.uy}}

\newcommand{\C}{\mathbb{C}}
\newcommand{\R}{\mathbb{R}}

\newcommand{\N}{\mathbb{N}}

\renewcommand{\d}{\mathrm{d}}

\newcommand{\g}{>}
\renewcommand{\l}{<}
\newcommand{\vol}{\text{vol}}

\newcommand{\manifolds}{\mathcal{M}}

\newcommand{\normal}{C_{\text{nor}}}
\newcommand{\transition}{C_{\text{tran}}}
\newcommand{\holonomy}{\text{hol}}

\newcommand{\gromovspace}{\mathcal{GH}}
\newcommand{\gromovdistance}{\mathrm{d}_{GH}}
\newcommand{\leafdistance}{d_{L}}

\newtheorem{theorem}{Theorem}[section]
\newtheorem{lemma}[theorem]{Lemma}
\newtheorem{corollary}[theorem]{Corollary}

\begin{document}

\maketitle

\begin{abstract}
We show that the Gromov-Hausdorff limit of a sequence of leaves in a compact foliation is a covering space of the limiting leaf which is no larger than this leaf's holonomy cover.  We also show that convergence to such a limit is smooth instead of merely Gromov-Hausdorff.  Corollaries include Reeb's local stability theorem, part of Epstein's local structure theorem for foliations by compact leaves, and a continuity theorem of Álvarez and Candel.  Several examples are discussed.
\end{abstract}

\tableofcontents

\section{Introduction}\label{introduction}

The Reeb local stability theorem \cite[Theorem 2]{reeb1947} states that if the fundamental group of a compact leaf in a foliation is finite then all nearby leaves are finite covers of it.  

It's apparent from the proof that, besides compactness, the key property of the leaf which yields stability isn't finiteness of its fundamental group, but finiteness of its holonomy group. This gives rise to the standard generalization given for example in \cite[pg. 70]{camacho-lins1985}.  

In the special case when a leaf is compact and has trivial holonomy one can strengthen the conclusion to yield that all nearby leaves are diffeomorphic to the given leaf (see \cite[Theorem 2]{thurston1974} where conditions under which one can guarantee trivial holonomy are discussed).

Another context in which a Reeb-type stability result appeared is in the study of compact foliations by compact leaves (by which we mean that all leaves are compact).

This family of foliations is surprisingly rich due to the fact that if the codimension is $3$ or more there may be leaves with arbitrarily large volume (see \cite{epstein-vogt1978}).  However, if one assumes that there is a uniform upper bound for the volume of all leaves then Epstein's local structure theorem \cite[Theorem 4.3]{epstein1976} yields that each leaf has a neighborhood consisting of leaves which are finite covers of it.

But, what can be said about stability of non-compact leaves?

For proper leaves in codimension one foliations some stability results in which one concludes that a leaf has a neighborhood consisting of leaves diffeomorphic to it have been obtained (see \cite{cantwell-conlon1981}, \cite{inaba1983} and the references therein).

However, both the partition into orbits and the center stable foliation of the geodesic flow of a compact hyperbolic surface are examples of foliations in which no leaf has this type of stability (in the first case because both periodic and non-periodic orbits are dense, in the second case because the center stable manifolds of periodic and non-periodic orbits are not diffeomorphic but both types of leaves are dense).

Informally one might define stability of a leaf as the property of having a neighborhood consisting of leaves which are `similar' to it.  The above examples suggest that, in order to obtain useful stability results which apply to recurrent non-compact leaves, the criterion used for measuring the similarity of two leaves should be weaker than diffeomorphism.

In \cite{alvarez-candel2003} Álvarez and Candel introduce the `leaf function' associating to each point in a foliation its leaf considered as a pointed metric space.  The codomain of this function is `Gromov space' which is the space of all pointed isometry classes of pointed proper metric spaces endowed with the topology of pointed Gromov-Hausdorff convergence.  They also present a program for studying the geometry (e.g. quasi-isometry invariants) of generic leaves in foliations.  One result in this program is that the leaf function of any compact foliation is continuous on the set of leaves (compact or otherwise) without holonomy (see \cite[Theorem 2]{alvarez-candel2003}).

In general a sequence of manifolds can converge in the Gromov-Hausdorff sense to a compact manifold without any element of the sequence being homeomorphic to the limit (for example one can shrink the handle on a sphere with one handle to obtain a sequence converging to a sphere, see \cite[Figure 7.4]{burago-burago-ivanov2001}).

However, families of manifolds having uniform curvature and injectivity radius bounds are known to be compact for stronger notions of convergence than Gromov-Hausdorff convergence (e.g. see \cite{anderson1990}).  Using such results one can conclude that on such families of manifolds Gromov-Hausdorff convergence is equivalent to a stronger form of convergence, and in particular that convergence of a sequence to a compact limit implies eventual diffeomorphism of the manifolds in the sequence to the limit.

Since the leaves of any compact foliation admit uniform curvature and injectivity radius bounds this shows that for compact foliations Álvarez and Candel's continuity theorem implies Reeb stability of compact leaves with trivial holonomy as a special case.

It seems natural to ask if further regularity properties of the leaf function might explain situations in which one knows that nearby leaves are covering spaces of a given leaf.

We will prove such regularity properties in Section \ref{maintheorems} below.

In particular we show that the leaf function of a compact foliation takes values in a compact subset of Gromov space on which Gromov-Hausdorff convergence coincides with the (a priori much stronger) notion of smooth convergence.

Furthermore the limit of a sequence of leaves is always a covering space of the `limiting' leaf and is itself covered by the holonomy cover of this leaf.

These results imply as corollaries Reeb's local stability theorem, the above-mentioned part of Epstein's local structure theorem, and Alvarez and Candel's continuity theorem.  Furthermore, they seem to clarify the behavior of leaves in concrete examples such as the Reeb transition.

Many interesting examples of foliations such as those coming from Control Theory (see \cite{sussmann1973} and \cite{stefan1974}) and the complex version of Hilbert's 16th problem (see \cite[Section 8]{ilyashenko2002}) have singularities.   For singular foliations of `Morse-Bott type' a Reeb-like stability result has been obtained in \cite[Theorem A]{scardua-seade2009}.  We note that the leaf function is still well defined for singular foliations and pointed Gromov-Hausdorff convergence (which allows manifolds to collapse) still seems relevant but our methods, which rely heavily on uniform curvature and injectivity radius bounds, fail completely.

Our article is organized as follows.  In Section \ref{definitionsandexamplessection} we define the family of foliations we will work on (we have chosen to work with maximal leafwise regularity and minimal transverse regularity), Gromov space, and the leaf function.  We also illustrate these concepts and our results with a series of examples.  In Section \ref{maintheorems} we state and prove the main theorems and discuss applications.  The rest of the article is devoted to building the tools used in the proofs of the main theorems.  The key points are the compactness theorem of Section \ref{uniformlyboundedgeometry} whose proof occupies the next several sections, the definition of the holonomy covering of Section \ref{holonomysection}, and the results on sequences of functions into foliations proven in Section \ref{leafwiseconvergencesection}.  In Section \ref{leafgeometrysection} we verify that the leaves of compact foliations do indeed admit uniform curvature and injectivity radius bounds.

\subsection*{Acknowledgements}

I would like to thank my advisors François Ledrappier and Matilde Martínez for their guidance, Liviu Nicolaescu, Karsten Grove, and Renato Bettiol for their help on geometry related questions,  Gábor Székelyhidi and Mehdi Lejmi for their seminar on Gromov-Hausdorff convergence which I attended to great benefit,  Rafael Potrie for providing the geometric interpretation of the Reeb transition foliation which I've illustrated below,  and Jesús A. Álvarez Lopez for several corrections and suggestions.

I'm also grateful to the University of Notre Dame for the time I've spent there while visiting François Ledrappier during which most of this work took form.

\section{Foliations, Gromov space and leaf functions}\label{definitionsandexamplessection}

In this section we define what we mean by a foliation (sometimes called a lamination by Riemannian leaves) and Gromov-space $\left(\gromovspace,\gromovdistance\right)$ which is a metric space whose elements are pointed isometry classes of proper metric spaces.   The leaf function of a foliation is a natural function into Gromov-space, we illustrate its regularity properties with several examples.

\subsection{Definitions}

By a $d$-dimensional foliation we mean a metric space $X$ partitioned into disjoint subsets called leaves.  Each leaf is assumed to be a continuously and injectively immersed $d$-dimensional connected complete Riemannian manifold.  We further assume that each $x \in X$ belongs to an open set $U$ such that there exists a Polish space $T$ and a homeomorphism $h:\R^d \times T \to U$ with the following properties:
\begin{enumerate}
 \item For each $t \in T$ the map $x \mapsto h(x,t)$ is a smooth injective immersion of $\R^d$ into a single leaf.
 \item For each $t\in T$ let $g_t$ be the metric on $\R^d$ obtained by pullback under $x \mapsto h(x,t)$ of the corresponding leaf's metric.  If a sequence $t_n$ converges to $t \in T$ then the Riemannian metrics $g_{t_n}$ converge smoothly on compact sets to $g_t$.
\end{enumerate}

Given a point $x$ in a foliation $X$ we denote by $(L_x,x,g_{L_x})$ the leaf of $x$ considered as a pointed Riemannian manifold with basepoint $x$.  We sometimes write only $L_x$ and leave the basepoint $x$ and metric $g_{L_x}$ implicit.  Homeomorphisms satisfying the conditions of $h$ above are called foliated parametrizations and their inverses are foliated charts.

We recall that in any metric space $(X,d)$ there is a natural distance between subsets, Hausdorff distance, which is defined by
\[d_H(A,B) = \inf\left\lbrace \epsilon \g 0: d(a,B) \l \epsilon\text{ and }d(A,b) \l \epsilon\text{ for all }a \in A\text{ and }b \in B\right\rbrace.\]

In what follows we use $B_r(x)$ to denote the open ball centered at a point $x$ in a metric space and $\overline{B_r}(x)$ to denote its closure.  A metric space is said to be proper if all closed balls are compact.

The Gromov-Hausdorff distance between two pointed proper metric spaces $(X_i,x_i,d_i)$ where $i=1,2$ is defined by
\[\gromovdistance\left(X_1,X_2\right) = \sum\limits_{n = 1}^{+\infty}2^{-n}\min\left(1,d_n\right)\left(X_1,X_2\right)\]
where
\[d_n\left(X_1,X_2\right) = \inf\left\lbrace d(x_1,x_2) + d_H(\overline{B_n}(x_1),\overline{B_n}(x_2))\right\rbrace\]
the infimum being taken over all distances $d$ on the disjoint union $\overline{B_n}(x_1) \sqcup \overline{B_n}(x_2)$ which coincide with $d_i$ when restricted to $\overline{B_n}(x_i)$ for $i=1,2$.

The notion of convergence induced by $\gromovdistance$ is called pointed Gromov-Hausdorff convergence which we abbreviate to $\gromovspace$-convergence.  For an introduction to this subject see \cite[Chapter 10]{petersen2006}.

The $\gromovspace$-distance satisfies the triangle inequality and is zero on a pair of spaces if and only if there is a pointed isometry between them.  Furthermore any proper pointed metric space is the $\gromovspace$-limit of a sequence of finite metric spaces.  Hence we may consider Gromov-space $\left(\gromovspace,\gromovdistance\right)$ which is the separable metric space obtained by endowing the set of pointed isometry classes of pointed proper metric spaces with the Gromov-Hausdorff distance (alternatively it's the metric completion of the set of isometry classes of finite pointed metric spaces with respect to $\gromovdistance$).

The leaf function of a foliation $X$ is the function from $X$ to $\gromovspace$ is defined by
\[x \mapsto L_x\]
where the leaf $L_x$ is considered up to pointed isometry.

We begin our study of the regularity of this function with a series of examples.

\subsection{Example: the vinyl record foliation}\label{vinylsection}

Consider a foliation of the closed annulus $\lbrace (x,y) \in \R^2: 1 \le x^2 + y^2 \le 2\rbrace$ such that the two boundary circles are leaves and all other leaves are spirals which accumulate on both boundary components.  The leaf function of such a foliation is clearly not continuous since there are leaves which are isometric to $\R$ accumulating on a leaf isometric to an Euclidean circle.

\begin{figure}[H]
\centering
\begin{minipage}{0.5\textwidth}
\includegraphics[width=\textwidth]{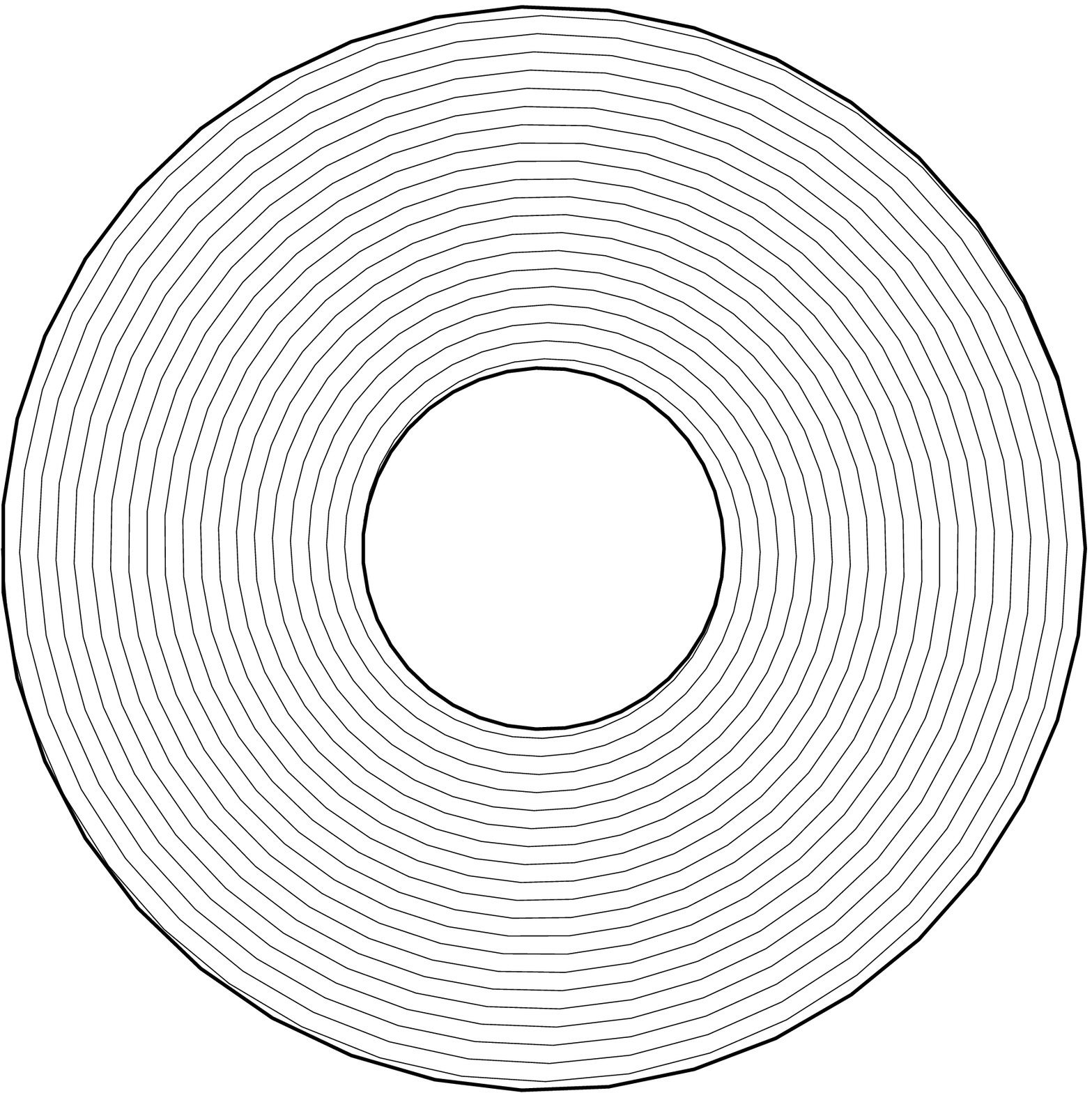}
\end{minipage}
\caption{The vinyl record foliation.}
\end{figure}

\subsection{Example: the Reeb cylinder}

Consider the foliation of the solid cylinder $C = \lbrace (x,y,z) \in \R^3: x^2 + y^2 \le \pi/2\rbrace$ where the boundary cylinder is a leaf and all other leaves are of the form $\lbrace (x,y,z) \in \R^d: z = t - \tan(x^2 + y^2)^2\rbrace$ for $t \in \R$.

In this example the leaf function is continuous but there are simply connected leaves accumulating on a non-simply connected leaf.   Hence the function
\[p \mapsto \widetilde{L}_p\]
associating to each point in $C$ the universal covering of its leaf, isn't continuous.

\begin{figure}[H]
\centering
\begin{minipage}{0.75\textwidth}
\includegraphics[width=\textwidth]{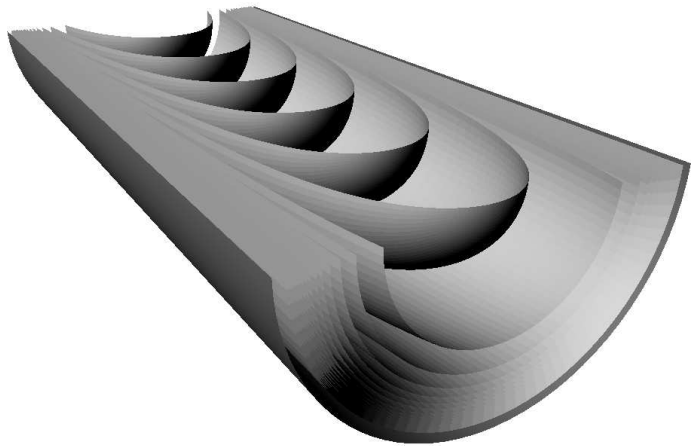}
\end{minipage}
\caption{A section of the Reeb cylinder.}
\end{figure}

\subsection{Example: the Reeb component}

One may take the quotient space of a Reeb cylinder by a translation along the axis to obtain a foliation of the solid torus normally called a Reeb component.

The leaf function of a Reeb component isn't continuous since for any sequence $x_n$ of interior points converging to a boundary point $x$ one has that the sequence of leaves $L_{x_n}$ converges to a cylinder $M$ while the leaf $L_x$ is a torus.

We notice that the cylinder $M$ is a covering space of the torus leaf $L_x$.  Furthermore one can choose a covering map from $M$ to $L_x$ in such a way that the image of the fundamental group of $M$ is exactly the set of curves in $L_x$ without holonomy.

Hence one sees that in this example the function
\[x \mapsto \widetilde{L_x}^{\holonomy}\]
associating to each point the holonomy covering of its leaf (see Section \ref{holonomysection}), is continuous.

\begin{figure}[H]
\centering
\begin{minipage}{0.75\textwidth}
\includegraphics[width=\textwidth]{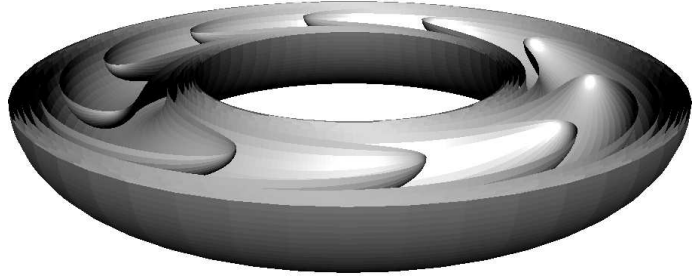}
\end{minipage}
\caption{Half a Reeb component.}
\end{figure}

\subsection{Example: the broken record foliation}

Consider a foliation of the closed annulus such that the inner boundary circle is accumulated on by trivially foliated annuli (i.e. annuli foliated by parallel circles) and also by copies of the vinyl record foliation.

The leaves inside the trivially foliated annuli have trivial holonomy and hence coincide with their holonomy covers.  However the inner boundary circle has non-trivial holonomy and hence its holonomy cover is isometric to $\R$.  This shows that the function
\[x \mapsto \widetilde{L_x}^{\holonomy}\]
isn't continuous.

\begin{figure}[H]
\centering
\begin{minipage}{0.5\textwidth}
\includegraphics[width=\textwidth]{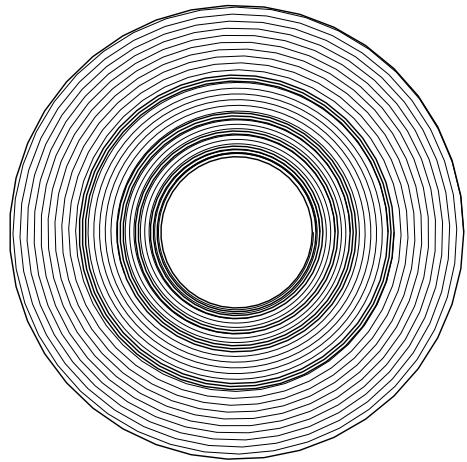}
\caption{A broken record foliation. Trivially foliated annuli and copies of the vinyl record foliation accumulate on the center circle.}
\end{minipage}
\end{figure}

\subsection{Example: the Reeb transition}\label{reebtransition}

The following example was introduced by Reeb in \cite{reeb1948}.

Consider the product Riemannian manifold $S^2 \times S^1 \times S^1$ where $S^2 = \lbrace (x,y,z) \in \R^2: x^2 + y^2 + z^2 = 1\rbrace$ is the standard two-dimensional sphere, and $S^1 = \lbrace z \in \C: |z| = 1\rbrace$ the standard circle.  We consider the coordinates $((x,y,z),e^{is}, e^{it})$ and the one forms
\[\left\lbrace \begin{array}{ll}\omega_1 &= \d t
\\ \omega_2 &= \left( (1-\sin(t))^2 + x^2\right) \d s + \sin(t) \d x   
  \end{array}\right.\]

The conditions of Frobenius' integrability theorem (see \cite[Theorem 2, pg. 185]{camacho-lins1985}) are satisfied and hence there is a unique two-dimensional foliation such that the tangent space of each leaf is contained in the kernel of $\omega_1$ and $\omega_2$.  The equation $\omega_1 = 0$ for vectors tangent to the foliation implies that each leaf is contained in a set of the form $S^2 \times S^1 \times \lbrace \text{constant}\rbrace$ and hence we may consider the foliation as a family of foliations on $S^2 \times S^1$ parametrized by $t$.

When $\sin(t) = 0$ it is easy to verify that one obtains the foliation of $S^2 \times S^1$ by leaves of the form $S^2 \times \lbrace \text{constant}\rbrace$.  However, when $\sin(t) = 1$ one has
\[\omega_2 = x^2 \d s + \d x\]
so that the torus in $S^2 \times S^1$ defined by $x = 0$ is a leaf, while the other leaves are planes parametrized by functions of the form
\[(x,y,z) \mapsto ((x,y,z), e^{i(c+1/x)})\]
on the hemispheres $x \l 0$ and $x \g 0$, for different values of the constant $c$.

Whenever $\sin(t) \neq 0$ one obtains a foliation of $S^2 \times S^1$ by spheres such that all leaves are obtained by applying a rotation to the $S^1$ components of a single leaf (i.e. they are all graphs of functions from $S^2$ to $S^1$ which in fact can be written explicitly).

Hence the set of spherical leaves is given by $\lbrace \sin(t) \neq 1\rbrace$, and the set of non-compact leaves is defined by $\lbrace \sin(t) = 1, x \neq 0\rbrace$.

One can explain this example geometrically.   By pasting two copies of the partition of the solid torus $D \times S^1$ into closed disks $D \times \lbrace\text{constant}\rbrace$ one can obtain the trivial foliation of $S^2 \times S^1$ by leaves of the form $S^2 \times \lbrace\text{constant}\rbrace$.  Pushing each disk at its center in the direction of the central circle of the solid torus one deforms the foliation but all leaves are still copies of $S^2$.  This is done in such a way that the number of turns each disk does around the solid torus diverges, at which point the boundary torus becomes a leaf and we obtain a foliation of $S^2 \times S^1$ by two Reeb components.  We call this process a Reeb transition.

Reeb noticed that in any such example there must be spherical leaves with arbitrarily large volume.  We will show that this is a consequence of the regularity properties of the leaf function.

\begin{figure}[H]\label{reebtransitionfigure}
\begin{minipage}{\textwidth}
\begin{minipage}{0.32\textwidth}
\includegraphics[width=\textwidth]{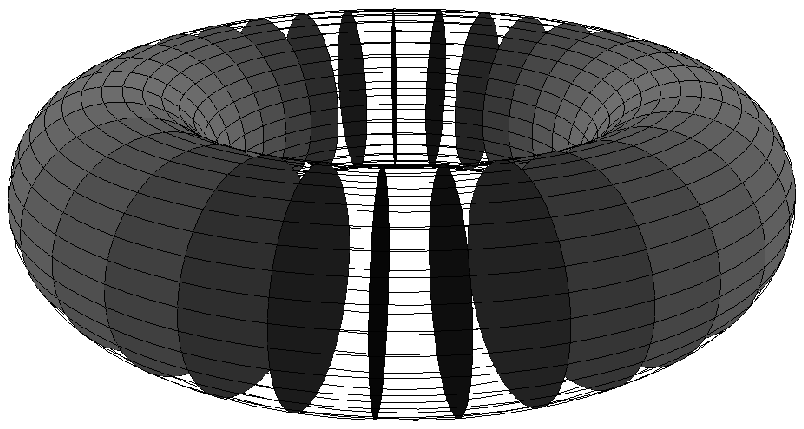}
\end{minipage}
\begin{minipage}{0.32\textwidth}
\includegraphics[width=\textwidth]{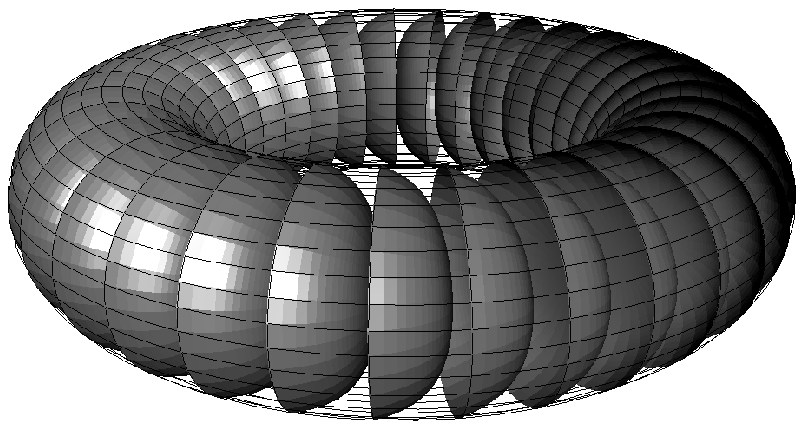}
\end{minipage}
\begin{minipage}{0.32\textwidth}
\includegraphics[width=\textwidth]{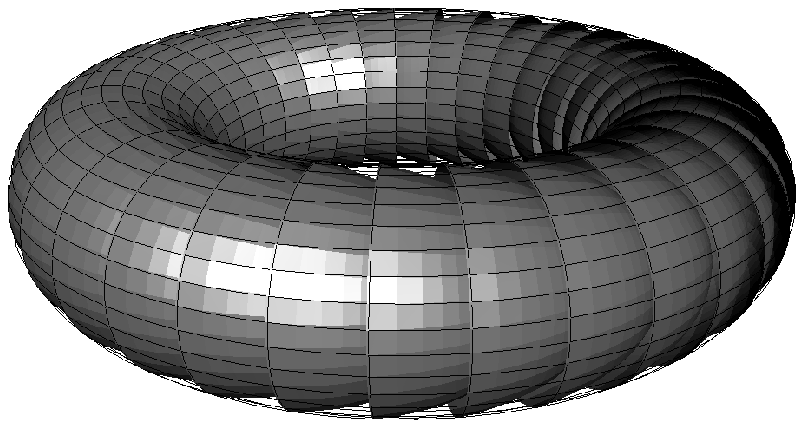}
\end{minipage}
\caption{A Reeb transition: the trivial partition of a solid torus into disks is deformed into a Reeb component.}
\end{minipage}
\end{figure}

\section{Regularity of leaf functions}\label{maintheorems}

In this section we state and prove our two main results after which we discuss applications to Reeb-type stability results and the Reeb transition example of the previous section.

\subsection{Regularity theorems}

A sequence of pointed complete connected Riemannian manifolds of the same dimension $(M_n,o_n,g_n)$ is said to smoothly converge to a pointed complete Riemannian manifold $(M,o,g)$ if there exists for each $r \g 0$  a sequence of pointed smooth embeddings $f_n:B_r(o) \to M_n$ of the open ball of radius $r$ centered at $o$ into $M_n$ defined for $n$ large enough with the property that the pullback Riemannian metrics $f_n^*g_n$ converge smoothly to $g$ on all compact subsets of $B_r(o)$ (see \cite[Chapter 10.3.2]{petersen2006} and Section \ref{topology}).

In principle smooth convergence of a sequence of manifolds is much stronger than $\gromovspace$-convergence of the same sequence.  However we can use compactness results from Riemannian geometry to obtain the following results.

\begin{theorem}[Precompactness of the leaf function]\label{foliationcompactnesstheorem}
Let $X$ be a compact $d$-dimensional foliation.  Then the leaf function of $X$ takes values in a compact subset $\manifolds$ of $\gromovspace$ which contains only complete Riemannian manifolds of dimension $d$.   Furthermore smooth and $\gromovspace$-convergence are equivalent on $\manifolds$.
\end{theorem}
\begin{proof}
We establish in Section \ref{leafgeometrysection} that there exists $r \g 0$ and a sequence $C_k$ such that the injectivity radius of all leaves is at least $r$ and the tensor norm of $k$-th derivative of the curvature tensor of any leaf is at most $C_k$.

Hence all leaves belong to the set $\manifolds$ of (isometry classes of) pointed complete $d$-dimensional Riemannian manifolds with geometry bounded by $(r,\lbrace C_k\rbrace)$ (see Section \ref{uniformlyboundedgeometry}).

We establish in Theorem \ref{compactnessofmanifoldspaces} that $\manifolds$ is $\gromovspace$-compact and that a sequence in $\manifolds$ converges smoothly if and only if it $\gromovspace$-converges.
\end{proof}

\begin{corollary}
If $x_n$ is a sequence converging to a point $x$ in a compact foliation $X$ and the sequence of leaves $L_{x_n}$ $\gromovspace$-converges to a pointed metric space $M$ then, in fact, $M$ is a smooth complete Riemannian manifold and $L_{x_n}$ converges smoothly to $M$.  In particular if $M$ is compact then $L_{x_n}$ is diffeomorphic to $M$ for all $n$ large enough.
\end{corollary}

By a Riemannian covering we mean a pointed local isometry $f:M \to N$ between complete pointed Riemannian manifolds.  If such a covering exists we say that $M$ is a Riemannian covering (or just a covering) of $N$ and that $N$ is covered by $M$.  See Section \ref{holonomysection} for the definition of the holonomy covering of a leaf.

\begin{theorem}[Semicontinuity of the leaf function]\label{semicontinuitytheorem}
Let $X$ be a compact foliation and $x_n$ be a sequence converging to a point $x \in X$.  If the sequence of leaves $L_{x_n}$ $\gromovspace$-converges to a pointed Riemannian manifold $M$ then $M$ is a Riemannian covering space of $L_x$ and is covered by $\widetilde{L_x}^{\holonomy}$.
\end{theorem}
\begin{proof}
By Theorem \ref{foliationcompactnesstheorem} the leaf function takes values in a compact subspace of $\gromovspace$ where Gromov-Hausdorff and smooth convergence are equivalent.  Hence $M$ is a complete Riemannian manifold and the sequence converges smoothly to $M$.

By smooth convergence (see Section \ref{topology}), for each $r \g 0$ there is a sequence of pointed embedding $f_{n,r}:B_r(o_M) \to L_{x_n}$ (defined for $n$ large enough) such that $|f_{n,r}^*g_{L_{x_n}} - g_M|_{g_M}$ converges uniformly to $0$ on $B_r(o_M)$.  We show in Lemma \ref{leafwiseconvergence} that this implies that the maps $f_{n,r}$ have a subsequence which converges locally uniformly to a local isometry $f_r:B_r(o_M) \to L_x$.

Now consider the family of functions $f_r:B_r(o_M) \to L_x$ when $r \to +\infty$.  Since all these functions are local isometries one obtains a local isometry $f:M \to L_x$ as a the uniform limit on compact subsets $f_{r_k}$ for some subsequence $r_k \to +\infty$.  Hence $M$ is a Riemannian covering of $L_x$ via the covering map $f$.

Suppose that for some pair of distinct points $x,y \in M$ one has $f(x) = f(y)$ and let $\alpha:[0,1] \to M$ be a curve joining $x$ and $y$.  Take $r \g 0$ large enough so that $B_r(o_M)$ contains $\alpha([0,1])$ and let $f_{n,r}:B_{r}(o_M) \to L_{x_n}$ be a sequence of embeddings as above which converges locally uniformly to $f$ on $B_r(o_M)$.

Since each $f_{n,r}$ is injective and the pullback metrics converge to $g_M$ the leafwise distance between $f_{n,r}(x)$ and $f_{n,r}(y)$ is bounded below by a positive constant for $n$ large enough.  However since $f_{n,r}\circ \alpha$ converges uniformly to $f\circ \alpha$ we obtain that the holonomy along the closed curve $f\circ \alpha$ is non-trivial (see Corollary \ref{nontrivialholonomy}).  

We have established that any closed curve in $L_x$ having a lift under $f$ which isn't closed has non-trivial holonomy.  In particular the lift of any curve with trivial holonomy in $L_x$ is closed in $M$ and hence the image of the fundamental group of $M$ under $f$ contains the subgroup of curves with trivial holonomy.  By the classification of covering spaces (see Lemma \ref{coverings}) $\widetilde{L_x}^\holonomy$ is a Riemannian cover of $M$.
\end{proof}

\subsection{Applications to continuity and Reeb stability}

The main result of \cite{epstein-millett-tischler1977} is that in any foliation the set of leaves without holonomy is residual.  Combined with Theorem \ref{semicontinuitytheorem} we obtain that the leaf function is continuous on a residual set.  Potential applications of this result to the study of quasi-isometry invariants of leaves are discussed by Álvarez and Candel in \cite[Section 2]{alvarez-candel2003}.

\begin{corollary}[Álvarez-Candel continuity theorem]\label{alvarez-candel}
The leaf function of any compact foliation is continuous on the set of leaves without holonomy.  In particular the set of continuity points contains a residual set.
\end{corollary}

Smooth convergence of a sequence to a compact manifold implies that the sequence elements are eventually diffeomorphic to the limit.  Combined with Theorem \ref{semicontinuitytheorem} one obtains Reeb's local stability theorem (see \cite[Theorem 2]{reeb1947}).

\begin{corollary}[Reeb's local stability theorem]
Let $X$ be a compact foliation and $x \in X$ be such that $\widetilde{L_x}$ is compact.  Then there exists a neighborhood $U$ of $x$ such that $\widetilde{L_y}$ is diffeomorphic to $\widetilde{L_x}$ for all $y \in U$.
\end{corollary}

The same argument gives the usual generalization of Reeb's stability theorem to compact leaves with trivial or finite holonomy (see for example \cite[pg. 70]{camacho-lins1985}).

\begin{corollary}[Stability of compact leaves with finite holonomy]\label{stabilityofcompacttrivialholonomy}
Let $X$ be a compact foliation and $x \in X$ be such that $\widetilde{L_x}^{\holonomy}$ is compact.  Then there is a neighborhood $U$ of $x$ such that for each $y \in U$ the leaf $L_y$ is compact and diffeomorphic to a covering space of $L_x$.
\end{corollary}

We say $X$ is a foliation by compact leaves if all leaves are compact.  The volume function of such a foliation is the function
\[x \mapsto \vol\left(L_x\right)\]
associating to each leaf its volume (which is finite).  Since a Riemannian covering has larger volume then the space it covers one obtains the following.

\begin{corollary}[Volume function semicontinuity]
Let $X$ be a compact foliation by compact leaves.  Then the volume function of $X$ is lower semicontinuous.
\end{corollary}

Notice that since any sequence of leaves has a smoothly convergent subsequence we obtain the following part of Epstein's structure theorem (see \cite[Theorem 4.3]{epstein1976}).

\begin{corollary}[Epstein]
Let $X$ be a compact foliation by compact leaves whose volume function is bounded.  Then every point $x \in X$ has a neighborhood $U$ such that for all $y \in U$ the leaf $L_y$ is diffeomorphic to a finite covering of $L_x$.
\end{corollary}

We say a foliation $X$ has codimension $k$ if it admits an atlas by foliated charts $\lbrace h_i:U_i \to \R^d \times T_i, i \in I\rbrace$ with $T_i = \R^k$ for all $i$.  Under this hypothesis $X$ is automatically a topological manifold.  

Notice that any holonomy transformations of a codimension one foliation will be a homeomorphism between two open subsets of $\R$.  We say a codimension one foliation is transversally orientable if every holonomy transformation associated to a closed chain of compatible charts is increasing.

The following elementary lemma implies that in a transversally oriented codimension one foliation by compact leaves all leaves have trivial holonomy (here we denote by $f^n(x) = f(f(\cdots f(x)\cdots))$ the $n$-th iterate of the point $x$ under the function $f$ and notice that in order for it to be well defined $f^k(x)$ must belong to the domain of $f$ for all $k = 0,\ldots,n-1$):
\begin{lemma}
Let $h: U \to V \subset \R$ be an increasing homeomorphisms between two neighborhoods of $0 \in \R$ such that $h(0) = 0$.  Then either $h$ is the identity  map or there exists $x \in U$ and $f = h^{\pm 1}$ such that the set $\lbrace f^n(x): n \ge 0\rbrace$ is well defined and infinite.
\end{lemma}

Epstein established in \cite{epstein1972} that a flow on a $3$-manifold for which all orbits are periodic has the property that the periods are bounded.  This was later generalized to state that compact codimension two foliations by compact leaves have bounded volume functions (see \cite{edwards-millett-sullivan1977}).  Notice that these results are very subtle since they are false for foliations of codimension $3$ or more (see \cite{epstein-vogt1978}).  The codimension one case follows directly from our results and the above elementary lemma.
\begin{corollary}
Let $X$ be a connected compact transversally oriented codimension one foliation.  Then the leaf function of $X$ is continuous.  In particular all leaves are diffeomorphic and the volume function is continuous.
\end{corollary}

For tranversally oriented codimension one foliations of connected manifolds Reeb's local stability combines with properties of one dimensional dynamics in the spirit of the lemma above to yield Reeb's global stability theorem (see \cite[Theorem 3]{reeb1947} and \cite[pg. 72]{camacho-lins1985}) which states that if a leaf has a compact universal cover than all leaves are diffeomorphic.

In view of these results one might conjecture that the set of leaves with compact universal cover, besides being open, is always closed.  However this is false as shown by the Reeb transition example given in Section \ref{reebtransition}.  We will now discuss some aspects of this example.

The fact that in the Reeb transition there must be spheres with arbitrarily large volume follows from Corollary \ref{alvarez-candel} and the smooth convergence given by Theorem \ref{foliationcompactnesstheorem}.  To see this consider a sequence of points $x_n$ belonging to compact leaves which converge to a point $x$ whose leaf is non-compact and notice that the sequence of manifolds $L_{x_n}$ smoothly converge to $L_x$.

Consider now in the same example a sequence $x_n$ on spherical leaves which converges to a point $x$ on the single torus leaf.  By Theorem \ref{semicontinuitytheorem} any smooth limit point of the sequence $L_{x_n}$ must either be a finite covering of the torus $L_x$ or the cylinder $\widetilde{L_x}^{\holonomy}$.  The first case is impossible because convergence to a compact limit would imply that the manifolds in the sequence $L_{x_n}$ are eventually diffeomorphic to the limit manifold which would have to be a torus.   Hence the sequence of spheres $L_{x_n}$ converges smoothly to the cylinder $\widetilde{L_x}^{\holonomy}$.

\section{Uniformly bounded geometry}\label{uniformlyboundedgeometry}

In this section we prove that certain subsets of $\gromovspace$ consisting of manifolds with `uniformly bounded geometry' are compact and that furthermore smooth and $\gromovspace$-convergence coincide on them.  This result was used in the proof of Theorem \ref{foliationcompactnesstheorem} and may also be of independent interest.

\subsection{Spaces of manifolds with uniformly bounded geometry}

We say a complete $d$-dimensional Riemannian manifold has geometry bounded by $r \g 0$ and a sequence $C_k$ if the injectivity radius of $M$ is at least $r$ at all points and the curvature tensor of $M$ satisfies
\[|\nabla^k R| \le C_k\]
for all $k$, where $\nabla$ denotes the covariant derivative and we are using the tensor norms induced by the Riemannian metric.

We use $\manifolds\left(d,r,\lbrace C_k\rbrace\right)$ to denote the subset of $\gromovspace$ consisting of all isometry classes of $d$-dimensional complete pointed Riemannian manifolds with geometry bounded by $r$ and the sequence $C_k$.

An element of $\manifolds\left(d,r,\lbrace C_k\rbrace\right)$ is represented by a triplet $(M,o_M,g_M)$ and two triplets represent the same element if there is a pointed isometry between them.   We will sometimes write $M \in \manifolds\left(d,r,\lbrace C_k\rbrace\right)$ in which case it is implied that the basepoint will be denoted by $o_M$ and the Riemannian metric by $g_M$.

\subsection{A smooth compactness theorem}

Usually $\gromovspace$-convergence of a sequence of manifolds is much weaker than smooth convergence.  However we will show they are equivalent on sets of manifolds with uniform bounded geometry.

To understand this it might be helpful to consider the following fact: Let $\mathcal{F}$ be a $C^1$ compact family of functions from the interval $[0,1]$ to $\R$. Then if a sequence $f_n$ in $\mathcal{F}$ converges uniformly to a limit $f$, in fact $f \in \mathcal{F}$ and the derivatives $f_n'$ converge uniformly to $f'$.

The proof can also be thought of as an application of the fact that a continuous bijective function whose domain is compact has a continuous inverse (in the setting of the previous paragraph the domain would be $\mathcal{F}$ with the $C^1$-topology the codomain would be the same set with the $C^0$-topology and function would be the identity).  The difficulty in our case is in establishing compactness of the domain plus a subtle technical point which is discussed immediately after the proof.

We will now state the main result of this section.

\begin{theorem}\label{compactnessofmanifoldspaces}
Let $\manifolds = \manifolds\left(d,r,\lbrace C_k\rbrace\right)$ for some choice of dimension $d$, radius $r$, and sequence $C_k$.  Then $\manifolds$ is a compact subset of $\gromovspace$ on which $\gromovspace$-convergence and smooth convergence are equivalent.
\end{theorem}
\begin{proof}[Proof with gap]
The proof rests on the following facts
\begin{enumerate}
 \item The set $\manifolds$ is precompact with respect to smooth convergence.
 \item The set $\manifolds$ is closed under smooth convergence.
 \item Smooth convergence implies pointed Gromov-Hausdorff convergence.
\end{enumerate}

We will establish facts 1 and 2 in sections \ref{precompactness} and \ref{closedness} respectively.

Fact 3 is generally accepted (e.g. see \cite[Section 10.3.2]{petersen2006} and \cite[Section 7.4.1]{burago-burago-ivanov2001}) but we include a proof in the next subsection for completeness.

Using these facts the proof proceeds as follows.  

Given a sequence $M_n$ in $\manifolds$ we may, using smooth precompactness, extract a smoothly convergent subsequence $M_{n_k}$ with limit $M$.  Since $\manifolds$ is closed under smooth convergence we have $M \in \manifolds$.  Finally, since smooth convergence implies pointed Gromov-Hausdorff convergence one has
\[\lim\limits_{n \to +\infty}\gromovdistance(M_{n_k},M) = 0.\]

This establishes that $\manifolds$ is a compact subset of $\gromovspace$.

Suppose now that some sequence $M_n$ in $\manifolds$ converges in the pointed Gromov-Hausdorff sense to $M \in \manifolds$.  Since any subsequence of $M_n$ will have a further subsequence which converges smoothly and any smooth limit must in fact coincide with $M$ we obtain that the original sequence $M_n$ converges smoothly to $M$.
\end{proof}

There is a gap in the above proof which is illustrated by the following example (see Figure \ref{asconvergenceisnottopological}).

Consider the sequence of functions indexed on finite strings of zeros and ones defined by
\[f_{a_1\ldots a_r}:[0,1] \to \R\]
\[f_{a_1\ldots a_r}(x) = \left\lbrace \begin{array}{ll}1&\text{ if }\sum\limits_{k = 1}^r a_k 2^{-k} \l x \l 2^{-r} + \sum\limits_{k = 1}^r a_k 2^{-k}.
\\ 0 &\text{ otherwise.}\end{array}\right.\]

The sequence doesn't converge Lebesgue almost surely to any function.  However any subsequence has a further subsequence which converges almost surely to $0$.  In particular the arguments in our proof above would imply that $L^2$ convergence and almost sure convergence coincide on the set of functions $\left\lbrace0\right\rbrace \cup \left\lbrace f_{a_1\ldots a_r}\right\rbrace$ but this conclusion is false.

To exclude this type of behavior it suffices to show that smooth convergence comes from a topology.  We do this in Section \ref{topology}.

\begin{figure}[H]
\begin{minipage}{\textwidth}
\begin{minipage}{0.32\textwidth}
\includegraphics[width=\textwidth]{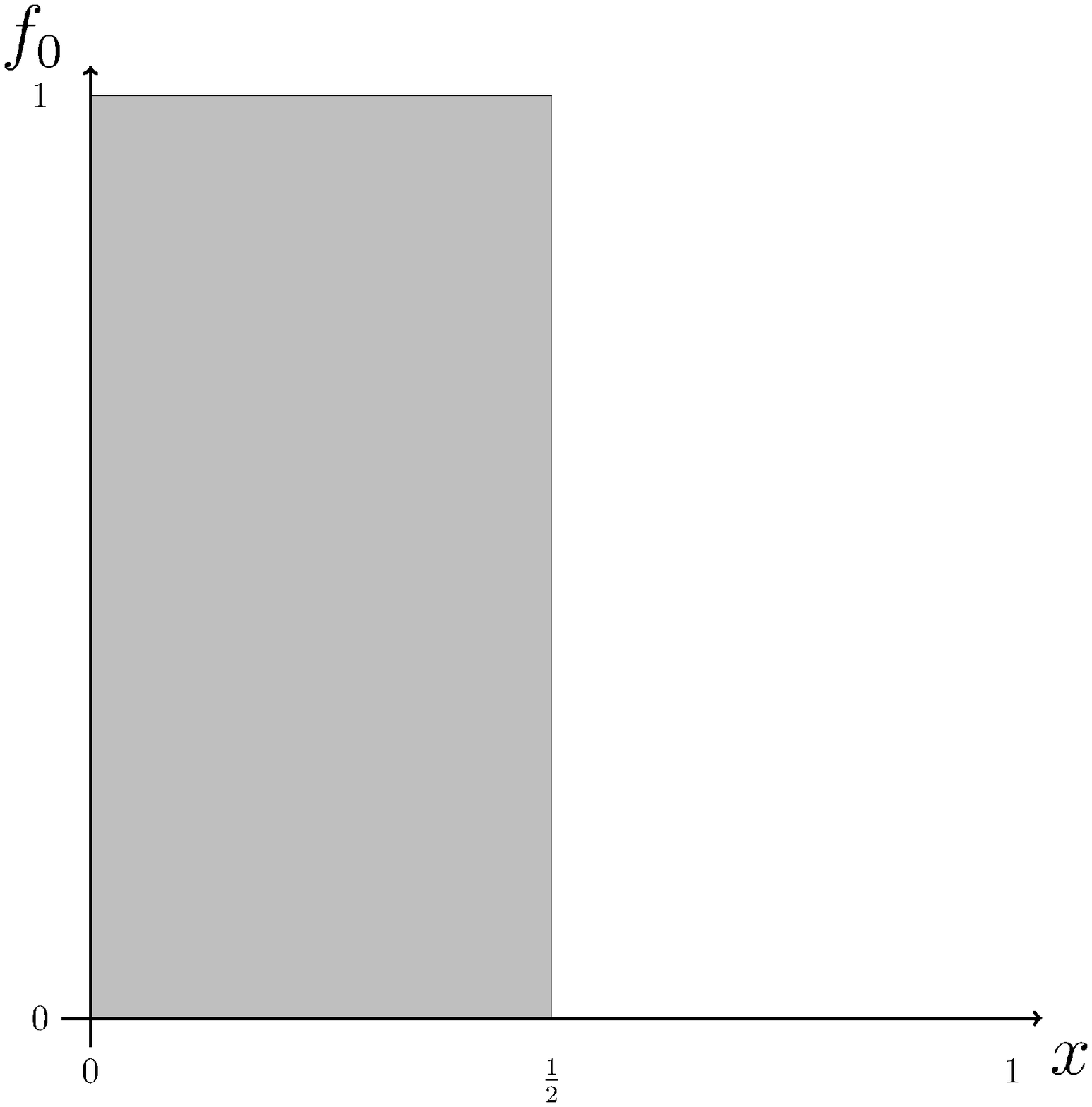}
\end{minipage}
\begin{minipage}{0.32\textwidth}
\includegraphics[width=\textwidth]{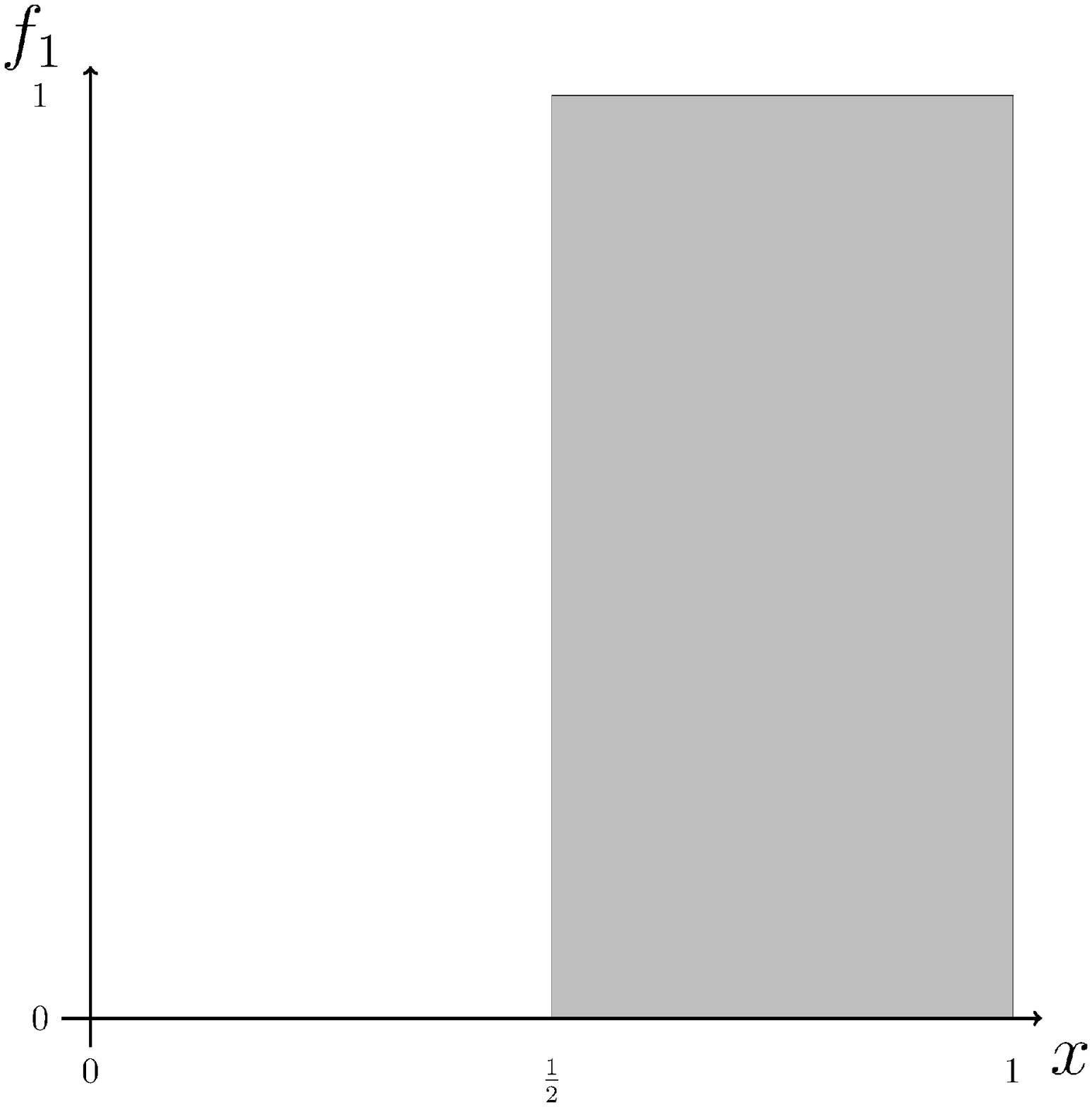}
\end{minipage}
\begin{minipage}{0.32\textwidth}
\includegraphics[width=\textwidth]{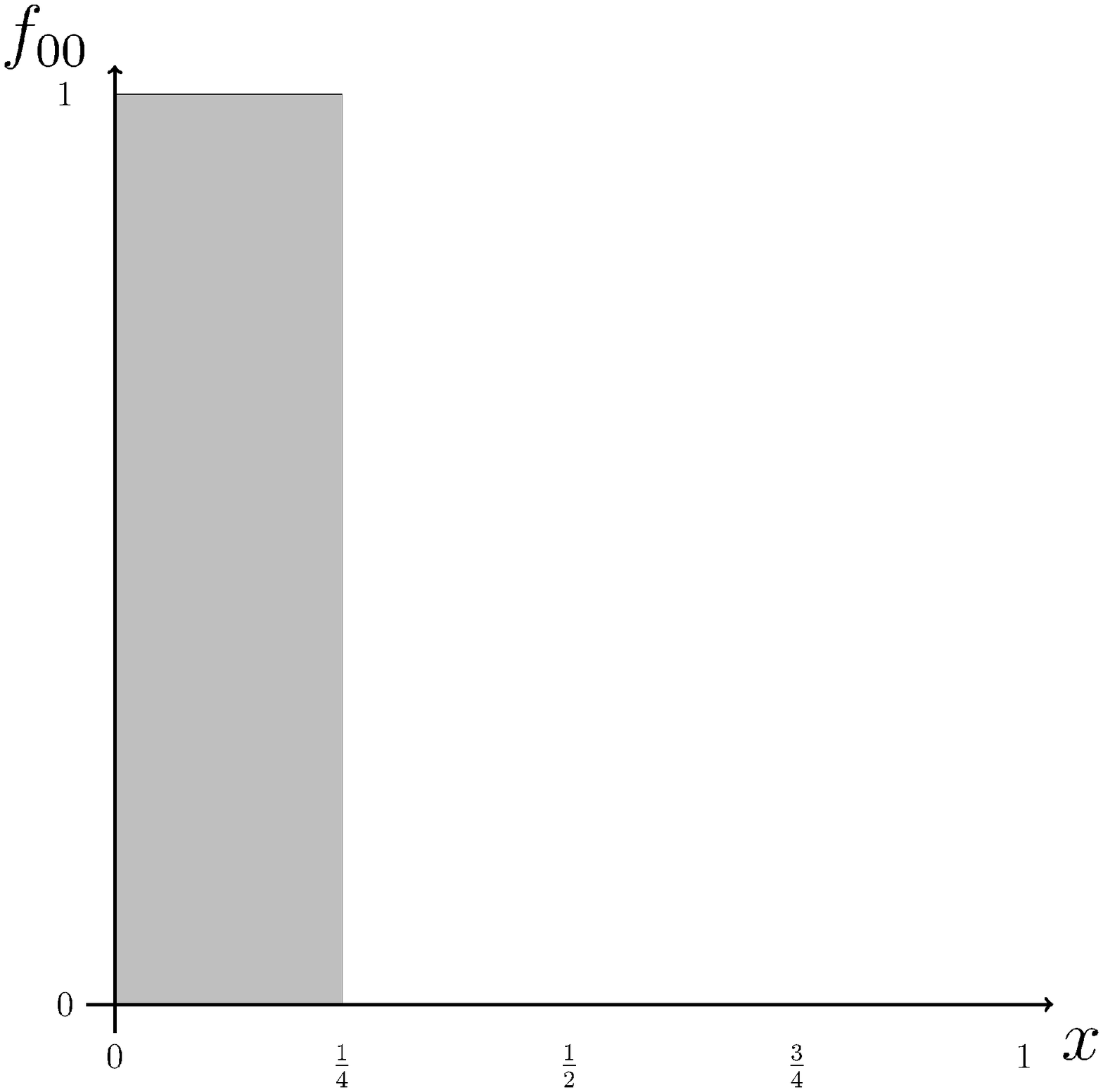}
\end{minipage}
\begin{minipage}{0.32\textwidth}
\includegraphics[width=\textwidth]{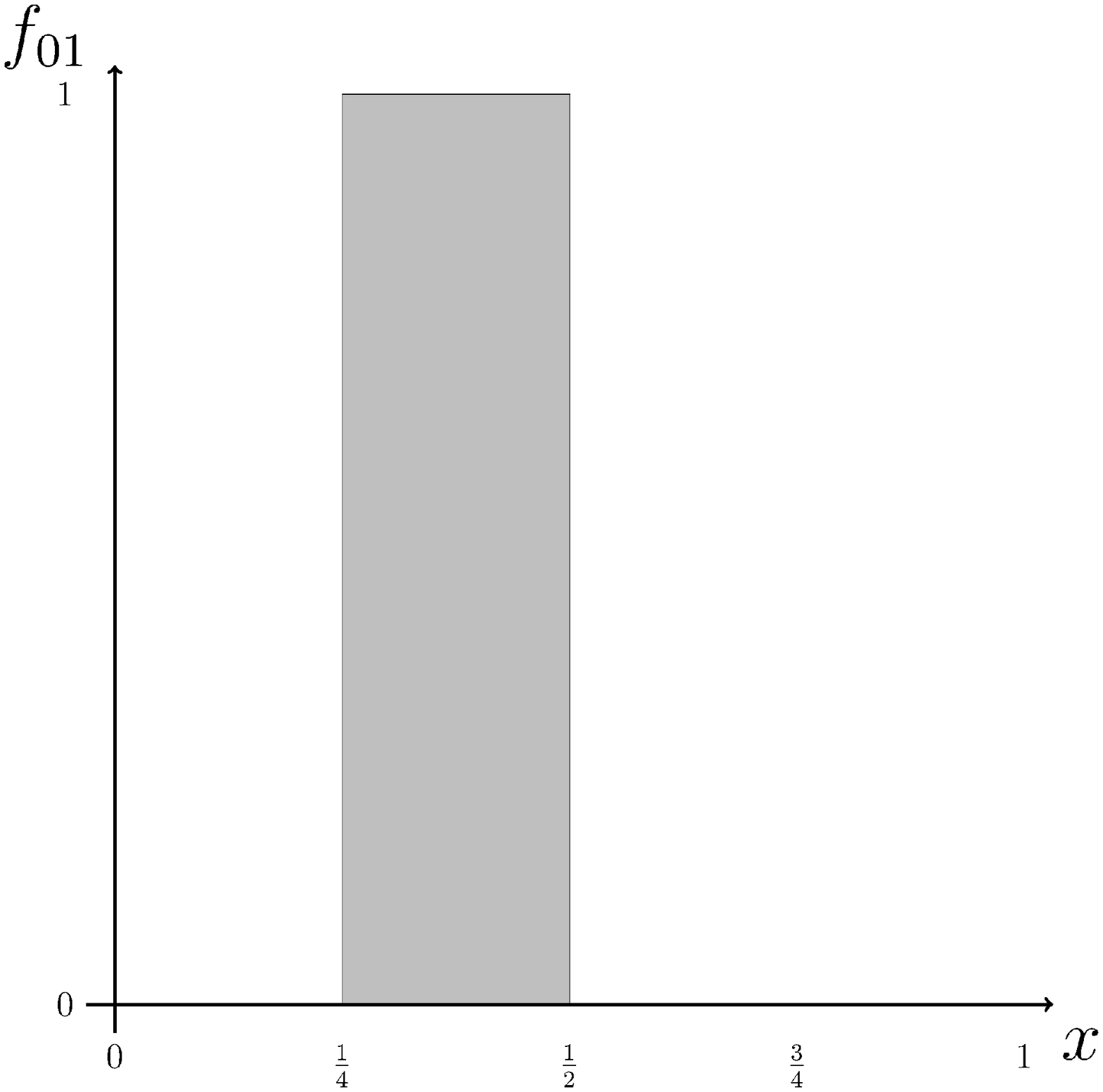}
\end{minipage}
\begin{minipage}{0.32\textwidth}
\includegraphics[width=\textwidth]{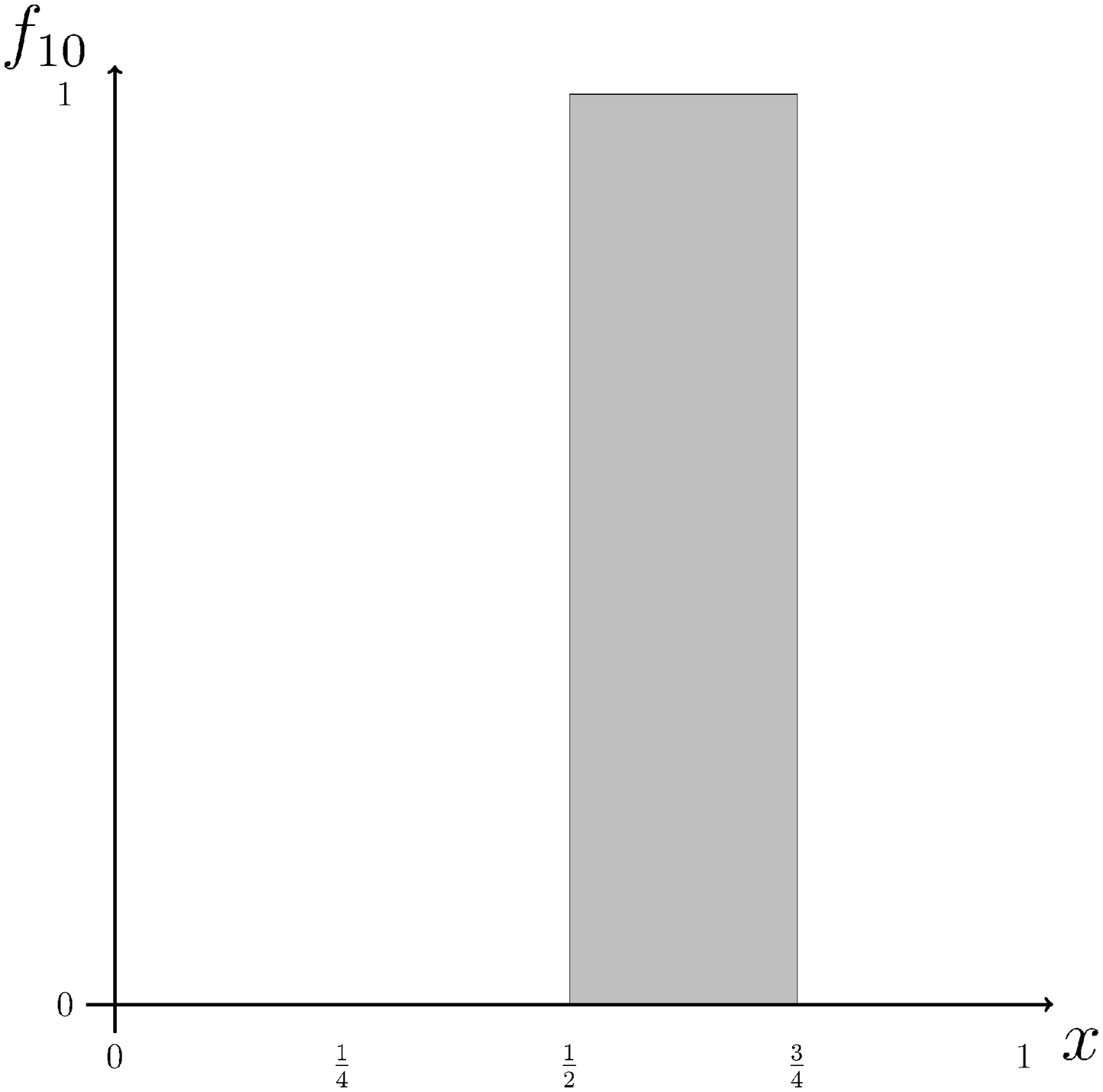}
\end{minipage}
\begin{minipage}{0.32\textwidth}
\includegraphics[width=\textwidth]{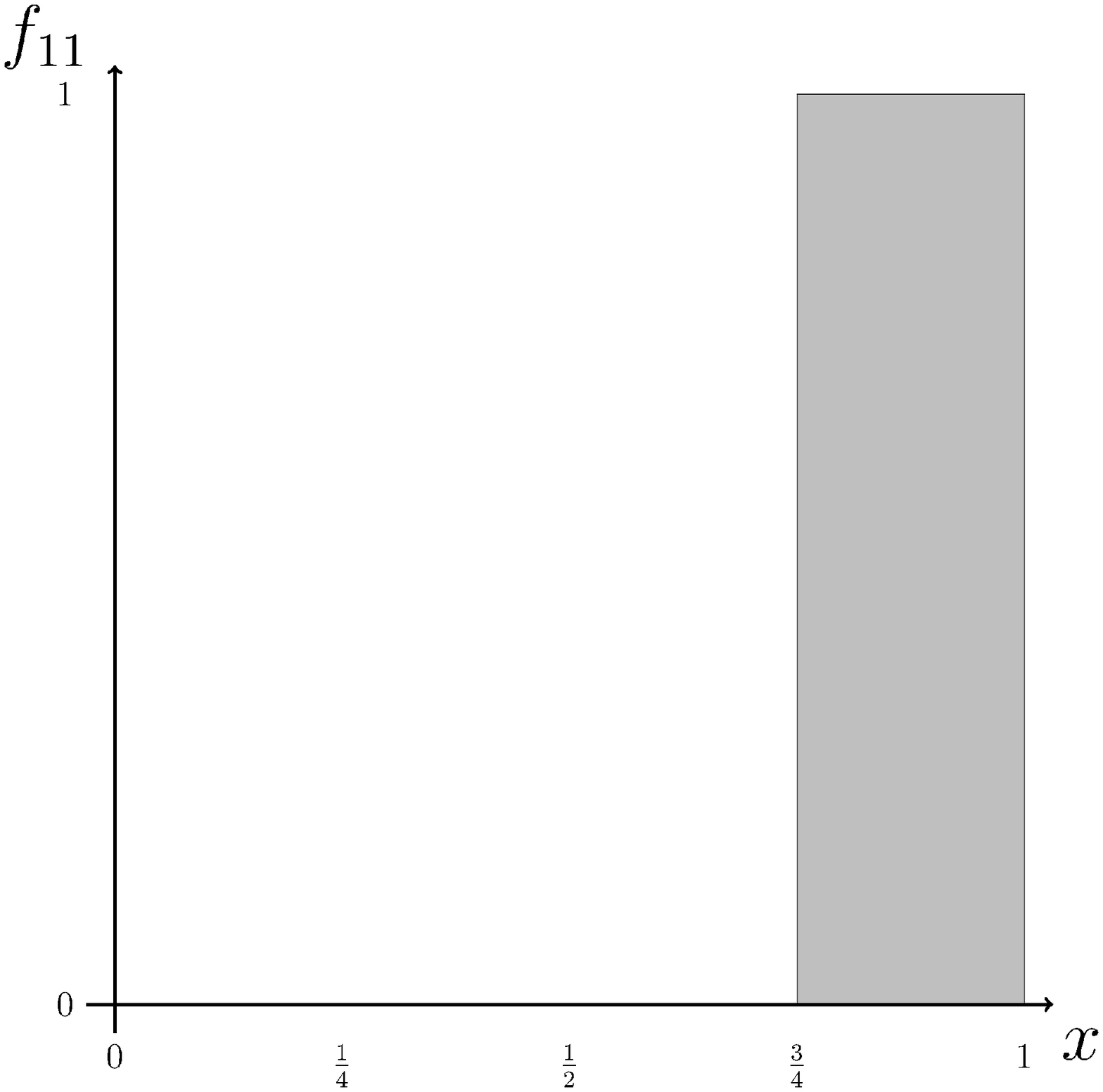}
\end{minipage}
\caption{Six elements of a sequence of functions which doesn't converge almost surely to $0$ but has no other limit points.}
\label{asconvergenceisnottopological}
\end{minipage}
\end{figure}

\subsection{Smooth vs Gromov-Hausdorff convergence}

For the readers convenience we present a proof of the fact that smooth convergence is stronger than $\gromovspace$-convergence.  The key ideas are contained in the proof of part 2 of \cite[Theorem 7.3.25]{burago-burago-ivanov2001} and the indications given in Section 7.4.1 of the same reference.

\begin{lemma}
If a sequence $(M_n,o_n,g_{M_n})$ converges smoothly to $(M,o,g)$ then it also $\gromovspace$-converges to the same limit.
\end{lemma}
\begin{proof}
We must show that for each $r \g 0$ the sequence of pointed compact metric spaces $\overline{B_r}(o_n)$ (where the metric is inherited from $M_n$) converges in the Gromov-Hausdorff sense to $\overline{B_r}(o)$.

By smooth convergence (see Section \ref{topology}) given $r \g 0$ there exists a sequence of smooth pointed embeddings $f_n:B_{3r}(o) \to M_n$ with the property that the pullback metrics $g_n = f_n^*g_{M_n}$ satisfy
\[a_n = \sup\lbrace |g_n(x) - g(x)|_g: x \in B_{3r}(o)\rbrace \to 0\]
when $n \to +\infty$.

Notice that whenever $a_n = 0$ one has that $\overline{B_r}(o)$ is isometric to $\overline{B_r}(o_n)$ via $f_n$ so that there is nothing to prove.  Hence we may assume without loss of generality in what follows that $a_n \neq 0$.  Also, since we are only interested in behavior when $n \to +\infty$ we may assume  that $a_n \l 1$.

Let $d$ be the Riemannian distance of $M$ and $d_n$ be the pullback under $f_n$ of the Riemannian distance on $M_n$.  Since the shortest curve between $f_n(x)$ and $f_n(y)$ might in principle exit $f_n(B_{3r}(o))$ it isn't necessarily true that $d_n$ equals the distance on $B_{3r}(o)$ induced by the metric $g_n$.

However notice that if $v$ is a tangent vector in $B_{3r}(o)$ of unit norm for $g$ then 
\[|g_n(v,v) - 1| \le a_n\]
so that the $g_n$ norm of $v$ is between $(1-a_n)^{1/2}$ and $(1+a_n)^{1/2}$.   This implies that the $g_n$-length of any curve in $B_{3r}(o)$ is within a multiplicative factor $b_n^{\pm 1}$ of its $g$-length where $b_n = \max\left( (1-a_n)^{-1/2},(1+a_n)^{1/2}\right)$.  In particular, for $n$ large enough, the Riemannian distance induced by $g_n$ on $B_{3r}(o)$ coincides with $d_n$ when restricted to $\overline{B_{r}}(o)$.

The previous comparison of lengths of curves also implies for $n$ large that
\[|d_n(x,y) - d(x,y)| \le (b_n-1)d(x,y) \le 2r(b_n - 1)\]
for all $x,y \in \overline{B_r}(o)$ (the first inequality relies on the fact that $1-b_n^{-1} \le b_n - 1$ which is true since $b_n \ge 1$).

Following the proof of part 2 of \cite[Theorem 7.3.25]{burago-burago-ivanov2001} we consider for each $n$ the distance $\tilde{d}_n$ on the disjoint union $\overline{B_r}(o) \sqcup \overline{B_r}(o)$ which coincides with $d$ on the left-hand copy, with $d_n$ on the right-hand copy and for $x,z$ in different copies is defined by
\[\tilde{d}_n(x,z) = \inf\left\lbrace d(x,y) + 2r(b_n-1) + d_n(y,z): y \in \overline{B_r}(o)\right\rbrace.\]

The Hausdorff distance between the two copies of $\overline{B_r}(o)$ with the above defined distance is less than $3r(b_n - 1)$ and therefore goes to $0$ when $n \to +\infty$.  This shows that the Gromov-Hausdorff distance between $\overline{B_r}(o)$ and $f_n(\overline{B_r}(o))$ (the later inheriting its metric from $M_n$) converges to $0$ when $n \to +\infty$.

To conclude it suffices to establish that the Hausdorff distance (with respect to the Riemannian distance on $M_n$) between $f_n(\overline{B_r}(o))$ and $\overline{B_r}(o_n)$ goes to $0$ when $n \to +\infty$.  This follows from our comparison of $d$ and $d_n$ since $f_n(\overline{B_r}(o))$ contains the ball of radius $b_n^{-1}r$ and is contained in the ball of radius $b_n r$ centered at $o_n$.
\end{proof}

\section{Smooth precompactness}\label{precompactness}

In this section we prove that sets of manifolds with uniformly bounded geometry are precompact with respect to smooth convergence. This was used in the proof of Theorem \ref{compactnessofmanifoldspaces}.  

We recall (see \cite[Chapter 10]{petersen2006} and Section \ref{topology}) that, in similar fashion to the definition of smooth convergence, a sequence of complete Riemannian manifolds $(M_n,o_n,g_n)$ is said to converge $C^k$ to $(M,o,g)$ if for each $r \g 0$ there exists a sequence of smooth pointed embeddings $f_n:B_r(o) \to M_n$ (defined for large enough $n$) such that the pullback metrics $f_n^*g_n$ converge $C^k$  to $g$ on compact subsets of $B_r(o)$.

\begin{lemma}\label{precompactnesslemma}
All subsets of $\gromovspace$ of the form $\manifolds = \manifolds\left(d,r,\lbrace C_k\rbrace\right)$ are sequentially precompact with respect to smooth convergence.
\end{lemma}
\begin{proof}
For each $M \in \manifolds$ we consider the atlas $\mathcal{A}$ by normal coordinates on the balls of radius $r'$ given by Lemma \ref{higherordercheeger} below.  

A theorem of Eichhorn (see Lemma \ref{eichhorn} below) shows that there exists a sequence $\normal^k$ such that all the metrics on $B_{r'}$ obtained from such coordinates have coefficients which satisfy
\[|\partial_{i_1}\cdots\partial_{i_k}g_{ij}| \le \normal^k\]
for all choices of indices $i_1,\ldots,i_k$.

Furthermore we establish in Lemma \ref{higherordercheeger} that there is a sequence $\transition^k$ bounding the $k$-th order partial derivatives of the transition maps of any such atlas $\mathcal{A}$ and that the Euclidean and Riemannian norms on $B_{r'}$ differ at most by a multiplicative factor of $2^{\pm 1/2}$.

This shows that for each $k$ there exists $Q$ such that all manifold in $\manifolds$ have $C^k$ norm less than or equal to $Q$ on a scale of $r$ in the sense of Petersen (see the definition in subsection \ref{petersennorms} below).

Applying Petersen's compactness theorem (see Theorem \ref{petersen} below) one obtains that $\manifolds$ is $C^{k}$ precompact for all $k$, and hence smoothly precompact as claimed.
\end{proof}

\subsection{Norms and sequential compactness}\label{petersennorms}

Following \cite[Chapter 10.3.1]{petersen2006} (taking, for simplicity, $\alpha =1$ in his notation) we say that a manifold $M$ has $C^k$-norm less than or equal to $Q$ on a scale of $r$ if there exists an atlas $\mathcal{A}$ of $M$ which satisfies the following properties:
\begin{enumerate}
 \item Every ball of radius $e^{-Q}r/10$ is contained in the domain of some chart in $\mathcal{A}$.
 \item For each chart $\varphi \in \mathcal{A}$ one has $|D\varphi| \le e^Q$ and $|D\varphi^{-1}| \le e^Q$, where $D\varphi$ is the tangent map to the chart and one uses the operator norm between the tangent space of $M$ with the Riemannian metric and Euclidean space with the usual Euclidean metric.
 \item For each chart $\varphi \in \mathcal{A}$ and each $0 \le i \le k$ the partial derivatives of order $i$ of the coefficients of $\varphi_*g_M$ are $Q/(r^{i+1})$-Lipschitz.
 \item For each $\varphi_1,\varphi_2 \in \mathcal{A}$ the $C^{k+2}$-norm (i.e. sum of suprema of absolute values of all partial derivatives up to order $k+2$) of the transition map $\varphi_2\circ\varphi_1^{-1}$ is less than or equal to $(10+r)e^Q$.
\end{enumerate}

We now restate Petersen's \cite[Theorem 72]{petersen2006} as we will use it.
\begin{theorem}[Petersen]\label{petersen}
For any positive constants $r$ and $Q$ the class of pointed, complete, $d$-dimensional Riemannian manifolds with $C^k$-norm less than or equal to $Q$ on a scale of $r$ is sequentially compact with respect to $C^{k}$ convergence.
\end{theorem}

\subsection{Normal coordinates}

We recall that a normal parametrization of a manifold $M \in \manifolds\left(d,r,\lbrace C_k\rbrace\right)$ at a point $p$ is a function $\psi: \R^{d} \to M$ satisfying
\[\psi(x) = \exp_{\psi(0)}\circ f(x)\]
where $\exp: T_{\psi(0)}M \to M$ is the Riemannian exponential map and $f:\R^{d} \to T_{\psi(0)}M$ is a linear isometry between $\R^{d}$ and the tangent space $T_{\psi(0)}M$ at $\psi(0)$.

If $M \in \manifolds\left(d,r,\lbrace C_k\rbrace\right)$ then any normal parametrization $\psi$ is a diffeomorphism when restricted to the ball $B_{r}$ of radius $r$ centered at $0 \in \R^{d}$.  Hence the pullback $g = \psi^*g_M$ of the Riemannian metric of $M$ to $B_{r}$ is also a Riemannian metric (i.e. non-degenerate).

We recall that the coefficients of a metric $g$ defined on some open subset of $\R^{d}$ are the functions
\[x \mapsto g(x)(e_i,e_j) = g_{ij}(x)\]
where $e_1,\ldots,e_{d}$ is the canonical basis of $\R^{d}$.

The coefficients obtained in this manner from manifolds in $\manifolds\left(d,r,\lbrace C_k\rbrace\right)$ are uniformly $C^k$ bounded as is shown by the following lemma (see \cite[Corollary 2.6]{eichhorn1991}).

\begin{lemma}[Eichhorn]\label{eichhorn}
Given $\manifolds = \manifolds\left(d,r,\lbrace C_k\rbrace\right)$ for each $k \ge 0$ there exists a constant $\normal^k$ such that if $g = \psi^*g_M$ is a metric on $B_{r}$ obtained by pulling back the metric of some manifold $M \in \manifolds$ via a normal parametrization $\psi$ then one has:
\[|\partial_{i_1}\cdots\partial_{i_{k}}g_{i j}| \le \normal^k\]
for all indices $i,j,i_1,\ldots,i_k$.
\end{lemma}

\subsection{Transition maps}

This subsection is devoted to establishing the following uniform estimate for the derivatives of transition maps between normal coordinates.

\begin{lemma}\label{higherordercheeger}
Given $\manifolds = \manifolds\left(d,r,\lbrace C_k\rbrace\right)$ there exists $r' \l r$ and for each $k \ge 0$ a constant $\transition^k$ such that the $k$-th order partial derivatives of any transition map between normal coordinates on balls of radius $r'$ in any manifold $M \in \manifolds$ are bounded in absolute value by $\transition^k$.
\end{lemma}

For partial derivatives of order one and two the above result can be compared to Lemma 3.4 and Lemma 4.3 of \cite{cheeger1970}.

The first derivative of the change of coordinates between maximal normal coordinates based at the north and south pole on the standard two dimensional sphere isn't bounded. This shows that it's indeed necessary to take $r' \l r$ in the above lemma.

Our proof proceeds in three steps.  First we bound the $k$-th order covariant derivative of any curve of the form $t \mapsto x + tv$ for any metric on the Euclidean ball of radius $r'$ in $\R^d$ obtained by pullback from a normal parametrization of a manifold in $\manifolds$.  Second, we bound the the actual (Euclidean) $k$-th order derivative of any curve whose covariant derivatives satisfy the previously obtained bounds (the point here being that covariant derivatives are invariant under the transition maps).  Finally, combining the preceding result one obtains a bound for the $k$-th derivative of any transition map along any straight line which implies the same bound is satisfied for the partial derivatives of order $k$ (this amounts to the statement that a symmetric $k$-linear function attains its maximum norm on the diagonal, see \cite{waterhouse1990} for a proof).

To begin we recall that the Christoffel symbols of a metric on an open subset of $\R^d$ with coefficients $g_{ij}$ are given by
\[\Gamma^k_{ij} = \frac{1}{2}g^{kl}\left(\partial_j g_{il} + \partial_i g_{lj} - \partial_l g_{ij}\right)\]
where $g^{ij}$ are the coefficients of the inverse of the matrix $(g_{ij})$ and summation is implied over the repeated indices of each term.

In what follows we use $B_s$ for the open Euclidean ball of radius $s$ centered at $0 \in \R^d$.

\begin{lemma}\label{christoffelbounds}
Given $\manifolds = \manifolds\left(d,r,\lbrace C_k\rbrace\right)$ there exists $r' \l r$ and for each $k \ge 0$ a constant $C_k'$ such that for any metric $g$ on $B_{r'}$ obtained by pullback from a normal parametrization of a manifold $M \in \manifolds$ one has:
\begin{enumerate}
 \item The $k$-th order partial derivatives of the metric coefficients $g_{ij}$, the coefficients of the inverse matrix $g^{ij}$, and the Christoffel symbols $\Gamma^l_{ij}$, are bounded in absolute value by $C_k'$ for all $k$.
 \item For all $v \in \R^d$ and $x \in B_{r'}$ one has $2^{-1}|v| \le |v|_{g(x)} \le 2|v|$ where $|v|$ is the Euclidean norm of $v$ and $|v|_{g(x)}$ its norm with respect to the inner product $g(x)$.
\end{enumerate}
\end{lemma}
\begin{proof}
Notice that for any of the coefficients $g_{ij}$ under consideration one has $(g_{ij}(0)) = (\delta_{ij})$ where the right-hand side is the $d \times d$ identity matrix.  Let $K$ be a compact neighborhood of the identity matrix such that any inner product whose matrix of coefficients (i.e. the matrix whose entry in the $i$-th row and $j$-th column is the inner product between the $i$-th and $j$-th vectors of the canonical basis of $\R^d$) is in $K$ satisfies property 2 above.

Since one has a uniform bound $\normal^1$ (given by Lemma \ref{eichhorn}) for the first order derivatives of $g_{ij}$ on $B_r$ there exits $r' \l r$ (depending only on this $\normal^1$) such that for all the metrics under consideration $(g_{ij}(x))$ belongs to $K$ for all $x \in B_{r'}$.

By Lemma \ref{eichhorn} one has uniform bounds on the partial derivatives of the metric coefficients $g_{ij}$ on $B_r$ (and in particular on $B_{r'}$). Combining this with the fact that matrix inversion is smooth on $K$ one obtains uniform bounds for the partial derivatives of all orders of the inverse matrix $(g^{ij})$ on $B_{r'}$. From this one can bound the partial derivatives of the Christoffel symbols as well.
\end{proof}

The covariant derivative of a vector field $v(t)$ over a curve $x(t)$ in $\R^d$ with respect to a metric with Christoffel symbols $\Gamma^k_{ij}$ is given by
\begin{equation}\label{covariantderivative}
\nabla_{x'}v = v' + \Gamma^k_{ij}(x^i)' v^je_k 
\end{equation}
where a superscript $i$ denotes the $i$-th coordinate and $'$ denotes derivative with respect to $t$.  We convene that $\nabla_{x'}^0v(t) = v(t)$ and define inductively $\nabla_{x'}^{k+1}v(t) = \nabla_{x'} \nabla_{x'}^kv(t)$.

\begin{lemma}\label{straightlinecovariantderivative}
Fix $\manifolds = \manifolds\left(d,r,\lbrace C_k\rbrace\right)$ and let $C_k'$ and $r'$ be given by Lemma \ref{christoffelbounds}.  There exists a sequence $C_k''$ such that for any metric $g$ on $B_{r'}$ obtained by pullback from a normal parametrization of a manifold $M \in \manifolds$ and any curve of the form
\[x(t) = x_0 + tv\]
where $x_0 \in B_{r'}$ and $|v| = 1$ one has
\[|\nabla_{x'}^kx'|_{g} \le C_k''\]
for all $k \ge 0$.
\end{lemma}
\begin{proof}
From Lemma \ref{christoffelbounds} the Riemannian norm of $v$ is bounded by $2$ at all points in $B_{r'}$.  This shows that one can take $C_0'' = 2$.

In order to bound the higher order covariant derivatives define inductively
\[v_2(t) = \nabla_{x'}v = v^i v^j \Gamma_{ij}^k e_k\]
and
\[v_{n+1}(t) = \nabla_{x'}v_n(t) = v_n' + v^iv_n^j\Gamma_{ij}^ke_k.\]

Since the coordinates  $v^i$ of $v$ are constants of absolute value less than or equal to $1$ the Euclidean norm of $v_{n+1}$ can be bounded in terms of that of $v_n$ and the derivatives of the Christofell symbols.  This is possible and is equivalent to bounding the Riemannian norm due to Lemma \ref{christoffelbounds}.
\end{proof}

We denote by $x^{(k)}(t)$ denote the $k$-th (Euclidean) derivative of a curve in $\R^d$.
\begin{lemma}
Fix $\manifolds = \manifolds\left(d,r,\lbrace C_k\rbrace\right)$ and let $C_k''$ and $r'$ be given by Lemma \ref{christoffelbounds}.  There exists a sequence $C_k'''$ such that for any metric $g$ on $B_{r'}$ obtained by pullback from a normal parametrization of a manifold $M \in \manifolds$ and any curve $x(t)$ satisfying
\[|\nabla_{x'}^kx'|_{g} \le C_k''\]
for all $k \ge 0$ one has
\[|x^{(k)}(t)| \le C_k'''\]
for all $k \ge 0$.
\end{lemma}
\begin{proof}
By Lemma \ref{christoffelbounds} the Euclidean and Riemannian norms differ at most by a factor of $2^{\pm1/2}$.

In particular one can take $C_0''' = 2 C_0''$ and the Euclidean norm of
\[\nabla_{x'}x' = x'' + \Gamma^k_{ij}(x^i)' (x^j)'e_k\]
is bounded by $2C_1''$.

Since one has $|x'| \le 2C_0''$ one obtains from the last equation a bound for $|x''|$.

The higher order case follows by induction since there is a single term in $\nabla_{x'}^kx'$ which is equal to $x^{(k+1)}$ and the rest can be bounded in terms of lower order derivatives of $x$ and the derivatives of the Christoffel symbols.
\end{proof}

We now complete the final step for the proof of Lemma \ref{higherordercheeger}.
\begin{lemma}
Let $f:U \subset \R^d \to \R^d$ be a smooth function satisfying
\[|g^{(k)}(0)| \le C_k''\]
for all $k \ge 0$ and $g$ of the form $g(t) = f(x + tv)$ with $|v| = 1$ and $x \in U$.  Then for all $x \in U$ one has
\[|\partial_{i_1}\cdots\partial_{i_k}f(x)| \le C_k''\]
for all $k \ge 0$ and $i_1,\ldots,i_k \in \lbrace 1,\ldots,d\rbrace$.
\end{lemma}
\begin{proof}
Define inductively
\[D_xf(v) = \lim_{h \to 0}\frac{f(x + hv) - f(x)}{h}\]
\[D^2_xf(v_1,v_2) = \lim_{h \to 0}\frac{D_{x+hv_1}f(v_2) - D_xf(v_2)}{h}\]
\[D^{k+1}_xf(v_1,\ldots, v_{k+1}) = \lim_{h \to 0}\frac{D_{x+hv_1}^k f(v_2,\ldots,v_{k+1}) - D_x^kf(v_2,\ldots, v_{k+1})}{h}.\]

Letting $P^k_xf(v) = D^k_xf(v,\ldots,v)$ we have by hypothesis and multilinearity that $|P^k_xf(v)| \le C_k^{''}|v|^k$.

Since partial derivatives commute the multilinear function $D^k_xf:(\R^d)^k \to \R^d$ is symmetric and $P^k_xf$ determines $D^k_xf$ by polarization.  This implies a bound for the mixed partial derivatives, and in fact one has $|D_x^kf(v_1,\ldots,v_k)| \le C_k^{''}|v_1|\cdots|v_k|$ as shown in \cite{waterhouse1990}.
\end{proof}

\section{Curvature and injectivity radius}\label{closedness}

In this section we prove that sets of manifolds with uniform bounded geometry are closed with respect to smooth convergence.  This was used in the proof of Theorem \ref{compactnessofmanifoldspaces}.  

\begin{lemma}
Suppose $\manifolds = \manifolds\left(d,r,\lbrace C_k\rbrace\right)$ for some value of the parameters.  If $(M_n,o_n,g_n)$ is a sequence in $\manifolds$ converging smoothly to $(M,o,g)$ then $M \in \manifolds$.
\end{lemma}
\begin{proof}
The fact that the injectivity radius of $M$ is larger than or equal to $r$ follows because the injectivity radius is upper semicontinuous with respect to smooth convergence as we will show in the next subsection (see Lemma \ref{injectivitysemicontinuity}).

We will now establish that $M$ satisfies the curvature bounds
\[|\nabla^k R|_g \le C_k.\]

Let $(g_{ij})$ be the matrix of coefficients of a metric $g$ on an open subset of $\R^{d}$ and $(g^{ij})$ the inverse matrix.  The $g$ norm of a $(p,q)$ tensor field\begin{footnote}{In all tensor calculations we use the convention that summation is implied over indices which are repeated in a term}\end{footnote}
\[T = a_{i_1,\ldots,i_q}^{i_{q+1},\ldots,i_{p+q}} e^{i_1}\otimes \cdots \otimes e^{i_q} \otimes e_{i_{q+1}} \otimes \cdots e_{i_{p+q}}\]
(where we denote by $e_i$ the canonical basis and $e^i$ the dual basis of $\R^{d}$) is given by
\[|T|_g^2 = a_{i_1,\ldots,i_q}^{i_{q+1},\ldots,i_{p+q}}a_{j_1,\ldots,j_q}^{j_{q+1},\ldots,j_{p+q}}g^{i_1j_1}\cdots g^{i_q j_q}g_{i_{q+1}j_{q+1}} \cdots g_{i_{p+q}j_{p+q}}.\]

The curvature tensor of $g$ is the $(1,3)$-tensor field $R = R_{ijk}^l e^i\otimes e^j \otimes e^k \otimes e_l$ given by (e.g. see \cite[Section 5]{brewin2010})
\[R_{ijk}^l = \partial_j \Gamma_{ki}^l - \partial_k \Gamma^l_{jk} + \Gamma_{jm}^k \Gamma^m_{ki} - \Gamma^l_{km}\Gamma^m_{ji}\]
where the Christoffel symbols $\Gamma^k_{ij}$ are defined by
\[\Gamma_{ij}^k = \frac{1}{2}g^{kl}(\partial_i g_{il} + \partial_j g_{jl} - \partial_l g_{ij}).\]

Since matrix inversion is smooth the two formulas above prove that if a sequence of metrics $g_n$ converges uniformly on compact sets to $g$ then the norm of their curvature tensors converge pointwise to that of $g$.

Similarly, for each $k$ the covariant derivative $\nabla^k R$ is a $(1,3+k)$-tensor field whose coefficients are smooth functions of the partial derivatives of the coefficients $g_{ij}$ and $g^{ij}$.  This shows that the bound $|\nabla^k R| \le C_k$ passes to the limit when a sequence of manifolds converges $C^\infty$ to another.  Hence one has that the limit manifold $M$ of the the sequence $M_n$ also satisfies these bounds.
\end{proof}

\subsection{Semicontinuity of the injectivity radius}

Continuity of the injectivity radius with respect to a varying family of metrics on a single compact manifold was established in \cite{ehrlich1974} and \cite{sakai1983}.  

The injectivity radius isn't continuous under smooth convergence of pointed manifolds as can be seen by considering a metric $g$ on $\R^2$ which has finite injectivity radius but is flat outside of a compact set.   In this setting the sequence of pointed manifolds $(\R^2, x_n, g)$ will smoothly converge to $\R^2$ endowed with the Euclidean metric if $x_n \to \infty$ when $n \to +\infty$.  Hence we have a sequence of manifolds with finite and constant injectivity radius converging to a manifold whose injectivity radius is infinite.

\begin{figure}[H]
\centering
\begin{minipage}{0.75\textwidth}
\includegraphics[width=\textwidth]{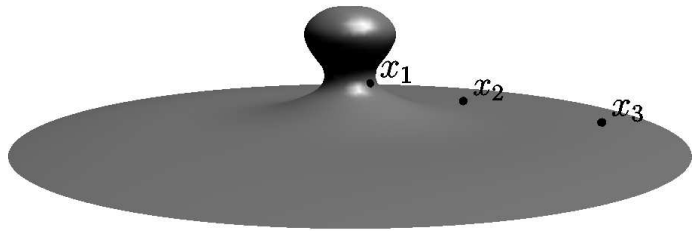}
\caption{An asymptotically flat surface with finite injectivity radius.  Changing the basepoint gives an sequence converging to a limit whose injectivity radius is infinite.}
\end{minipage}
\end{figure}

However, upper semicontinuity still holds as we will now show.

\begin{lemma}[Semicontinuity of the injectivity radius]\label{injectivitysemicontinuity}
The injectivity radius is upper semicontinuous with respect to smooth convergence.
\end{lemma}
\begin{proof}
Suppose for the sake of contradiction that there is a sequence $(M_n,o_n,g_n)$ with the injectivity radius of each term larger than or equal to some $r \g 0$ which converges smoothly to a manifold $(M,o,g)$ whose injectivity radius is strictly less than $r$.

By Proposition 19 and Lemma 14 of \cite[pg. 139-142]{petersen2006}, there exists a geodesic $\alpha:[0,1] \to M$ of length $L \l r$ and some other smooth curve $\beta:[0,1] \to M$ with the same endpoints with length $L' \l L$.

By the definition of smooth convergence there is an open set $\Omega$ containing $\alpha([0,1])$ and $\beta([0,1])$ and a sequence of pointed embeddings $f_n: \Omega \to M_n$ such that $f_n^*g_n$ converges $C^{\infty}$ to $g$ on compact subsets of $\Omega$.

Consider $\alpha_n:[0,1] \to M$ the geodesic for the metric $f_n^*g_n$ with initial condition $\alpha'(0)$.  We claim that $\alpha_n(1) \to \alpha(1)$ and that the $f_n^*g_n$ length of $\alpha_n$ converges to $L$ when $n \to +\infty$.  By covering $\alpha([0,1])$ with a finite number of charts and noticing that in each chart the coefficients of $f_n^*g_n$ converge $C^\infty$ on compact sets to those of $g$, this follows from continuity of solutions to ordinary differential equations with respect to the vector field (see \cite[Theorem B3, pg. 333]{duistermaat-kolk2000}).  We omit further details.

Smooth convergence of $f_n^*g_n$ to $g$ implies that the $f_n^*g_n$ length of $\beta$ converges to $L'$ and the $f_n^*g_n$ distance between $\beta(1)$ and $\alpha_n(1)$ converges to $0$.

Hence for $n$ large enough the manifold $M_n$ contains a geodesic of length strictly less than $r$ which is not the shortest curve between its endpoints.  By the Hopf-Rinow theorem we will find two geodesics of length strictly less than $r$ joining the same endpoints in $M_n$ contradicting the fact that the injectivity radius of $M_n$ is larger than or equal to $r$.
\end{proof}

\section{Smooth convergence and tensor norms}\label{topology}

In this section we discuss in detail the definition of $C^k$ and smooth convergence of pointed Riemannian manifolds.  In particular we provide a coordinate free definition (which has been used throughout the article) of convergence in terms of tensor norms and prove that it is equivalent to definition given in \cite[Chapter 10.3.2]{petersen2006}.  

We also establish that $C^k$ and smooth convergence on certain subsets of $\gromovspace$ comes from a topology, a fact that was used in the proof of Theorem \ref{compactnessofmanifoldspaces}.

\subsection{Coordinate free definition of convergence}

Following \cite[10.3.2]{petersen2006} a sequence $(M_n,o_n,g_n)$ of pointed connected complete Riemannian manifolds is said to converge $C^k$ to $(M,o,g)$ if for every $r \g 0$ there exists a domain $\Omega$ containing $B_r(o)$ and (for $n$ large enough) a sequence of pointed embeddings $f_n:\Omega \to M_n$ such that $f_n(\Omega) \supset B_r(o_n)$ and $f_n^*g_n$ converges $C^k$ to $g$ on compact subsets of $\Omega$.  Smooth convergence is by definition $C^k$ convergence for all $k$.

Recall that the coefficients of a Riemannian metric $g$ defined on an open subset $U$ of $\R^d$ are the functions
\[x \mapsto g(x)(e_i,e_j) = g_{ij}(x)\]
where $e_1,\ldots,e_d$ is the canonical basis of $\R^d$.

By $C^k$ convergence of $f_n^*g_n$ to $g$ on compact subsets of $\Omega$ we mean that for any smooth parametrization $h:U \to V \subset \Omega$ the coefficients of the metrics $h^*f_n^*g_n$ converge to those of $h^*g$ in the $C^k$ topology on every compact subset of $U$.

To see that the restriction to compact subsets of $U$ is necessary consider the sequence of Riemannian metrics $g_n$ on the open interval $(0,1)$ defined by
\[g_n(x)(v,w) = e^{x/n}vw.\]

The sequence of coefficients $x \mapsto e^{x/n}$ in this example converges uniformly to the coefficient of the metric $g$ on $(0,1)$ given by
\[g(x)(v,w) = vw\]
however taking pullback under the diffeomorphism $h:(0,1) \to (0,1)$ defined by $h(x) = x^{\alpha}$ one obtains
\[h^*g_n(x)(v,w) = e^{x^{\alpha}/n}\alpha^2x^{2(\alpha-1)}vw\]
so that taking for example $\alpha = 1/2$ one sees that uniform convergence of the sequence of coefficients no longer holds.

We now present a coordinate free definition of $C^k$ convergence.

For this purpose we recall that a $(p,q)$ tensor on a vector space $V$ is an element of $(V^*)^{\otimes q} \otimes V^{\otimes p}$.  If $g$ is an inner product on $V$ then $g$ induces an inner product and norm on the space of $(p,q)$ tensors.  This inner product can be defined by taking any $g$-orthonormal basis $v_1,\ldots,v_d$ of $V$, considering the dual basis $v^1,\ldots, v^d$,  and declaring that the tensors of the form $v^{i_1}\otimes \cdots v^{i_q}\otimes v_{i_{1+q}} \otimes \cdots \otimes v_{i_{p+q}}$ are orthonormal.

In particular given a Riemannian manifold $(M,g)$ and a $(p,q)$ tensor field $T$ one can consider the tensor norm $|T(x)|_g$ of the tensor $T(x)$ over the tangent space $T_xM$ with respect to the inner product $g(x)$.

\begin{lemma}[Characterization of convergence]\label{coordinatefreeconvergence}
A sequence $(M_n,o_n,g_n)$ of pointed connected complete Riemannian manifolds converges $C^k$ to $(M,o,g)$ if and only if for each $r \g 0$ there exists a sequence of pointed embeddings (defined for $n$ large enough) $f_n:B_r(o) \to M_n$ such that
\[\lim\limits_{n \to +\infty}\sup\lbrace |\nabla^i (f_n^*g_n - g)(x)|_g: x \in B_r(o), i = 0,\ldots,k\rbrace \to 0\]
where $\nabla$ denotes the covariant derivative corresponding to the Riemannian metric $g$ (in particular for $i \neq 0$ one has $\nabla^i g = 0$).
\end{lemma}
\begin{proof}
Assume first that a sequence $(M_n,o_n,g_n)$ in $\manifolds$ converges $C^k$ to $(M,o,g)$ and fix $r \g 0$.

By definition of $C^k$ convergence there exists a domain $\Omega \supset B_{2r}(o)$ sequence of pointed embeddings $f_n:\Omega \to M_n$ such that $f_n^*g_n$ converges $C^k$ on compact sets of $\Omega$ to $g$.  This means that in any local chart the coefficients of $f_n^*g_n$ will converge $C^k$ on compact sets to those of $g$.  By Lemma \ref{tensorconvergence} below this implies $|\nabla^i(f_n^*g_n - g)| \to 0$ uniformly on compact subset of $B_{2r}(o)$ for $i = 0,\ldots, k$.  In particular since $\overline{B_r}(o)$ is compact one has
\[\lim\limits_{n \to +\infty}\sup\lbrace |\nabla^i g_n(x) -\nabla^i g(x)|_g: x \in B_r(o), i = 0,\ldots,k\rbrace \to 0\]
as claimed.

We will now prove the converse claim.

Given $r$ we must obtain a sequence of embeddings $f_n$ of an open set $\Omega \supset B_r(o)$ into $M_n$ such that $f_n(\Omega) \supset B_r(o_n)$ and $f_n^*g_n$ converges $C^k$ to $g$ on compact sets.  We will show that one can take $\Omega = B_{2r}(o)$.

By hypothesis there exists a sequence of pointed embeddings $f_n:B_{2r}(o) \to M_n$ such that $|\nabla^i(f_n^*g_n - g)| \to 0$ uniformly for $i = 0,\ldots,k$.  By Lemma \ref{tensorconvergence} below this implies that $f_n^*g_n$ converges $C^k$ to $g$ on compact subsets of $B_{2r}(o)$.

We must now establish that $f_n(B_{2r}(o)) \supset B_r(o_n)$ for all $n$ large enough.

To see this let $v_1,\ldots,v_{d}$ be a $g$-orthonormal basis of the tangent space $T_xM$ at a point $x \in B_{2r}(o)$ and $v^1,\ldots,v^{d}$ the dual basis.  One has $f_n^*g_n(x) = a_{ij} v^i\otimes v^j$ and
\[|(f_n^*g_n - g)(x)|_g^2 = \sum\limits_{i,j} (a_{ij}-\delta_{ij})^2\]
where $(\delta_{ij})$ is the identity matrix.  

For all $n$ large enough the left-hand side above will be small enough to guarantee that $a_{11} = f_n^*g_n(x)(v_1,v_1) = |v_1|_{f_n^*g_n}^2 \g 1/4$.  And, since one can choose any $g$-orthonormal basis to calculate the norm above, this implies
\[\frac{1}{2}|v|_{g} \l |v|_{f_n^*g_n}\]
for all $v \in T_xM$ and all $x \in B_{2r}(o)$.

In particular for $n$ large enough the $f_n^*g_n$ length of any curve joining $o$ and the boundary of $B_{2r}(o)$ will be at least $r$.  So that $f_n(B_{2r}) \supset B_r(o_n)$ as claimed.
\end{proof}

The following consequence was used in the proof of Theorem \ref{compactnessofmanifoldspaces}.

\begin{lemma}\label{topologylemma}
On any subset of $\gromovspace$ of the form $\manifolds = \manifolds\left(d,r,\lbrace C_k\rbrace \right)$ smooth convergence is topologizable.
\end{lemma}
\begin{proof}
We define the $k$-th order $(r,\epsilon)$-neighborhood of a manifold $M \in \manifolds$ as the set of $N \in \manifolds$ such that there exists a pointed embedding $f:B_r(o_M) \to N$ satisfying
\[\sup\left\lbrace |\nabla^i (g_M - f^*g_N)(x)|_g: x \in B_r(o_M), i = 0,\ldots, k\right\rbrace \l \epsilon.\]

By Lemma \ref{coordinatefreeconvergence} convergence with respect to the topology on $\manifolds$ generated by all $k$-th order $(r,\epsilon)$-neighborhoods (for all $k \in \N, r \g 0$ and $\epsilon \g 0$) coincides with smooth convergence.
\end{proof}

\subsection{Convergence of tensor fields}

We will now complete the calculations in local coordinates needed for the proof of Lemma \ref{coordinatefreeconvergence}.

Recall that the coefficients of a $(p,q)$ tensor field $T$ on an open set of $\R^{d}$ are the functions
\[x \mapsto T(x)(e_{i_1},\ldots,e_{i_q},e^{i_{q+1}},\ldots, e^{i_{p+q}}).\]

In what follows we use $|T|$ to denote the Euclidean tensor norm and $|T|_g$ to denote the tensor norm coming from a metric $g$.  

The following result characterizes $C^k$ convergence of the coefficients of such a tensor on a compact set in a coordinate invariant manner.

\begin{lemma}[Convergence of tensor fields]\label{tensorconvergence}
 Let $U$ be an open subset of $\R^d$, $g$ a Riemannian metric on $U$, $K$ a compact subset of $U$, and $T_n$ a sequence of $(p,q)$-tensor fields on $U$.  Then the following two statements are equivalent for all $k \ge 0$:
 \begin{enumerate}
  \item The coefficients of $T_n$ and their partial derivatives up to order $k$ converge to $0$ uniformly on $K$.
  \item For each $i = 0,1,\ldots, k$ one has
  \[\lim\limits_{n \to +\infty} \max\left\lbrace |\nabla^i T_n(x) |_g : x \in K \right\rbrace = 0.\]
 \end{enumerate}
\end{lemma}
\begin{proof}
Since $K$ is compact and the metric coefficients and Christoffel symbols are smooth there exist constants $C \ge 1$ and $\Gamma \g 0$ such that 
\begin{enumerate}
 \item The absolute value of the derivatives of the Christoffel symbols up to order $k$ are bounded by $\Gamma$ on $K$.
 \item For any tensor field $T$ of type $(p,q')$ with $q \le q' \le q+k$ one has
 \[C^{-1}|T(x)| \le |T(x)|_g \le C|T(x)|\]
 for all $x \in K$.
\end{enumerate}

Notice that, by the existence of the constant $C$ above, if $T_n$ is a sequence of $(p,q')$-tensor fields with $q \le q' \le q+k$ then $|T_n(x)|$ converges to $0$ uniformly on $K$ if and only if $|T_n(x)|_g$ does.  On the other hand $|T_n(x)|$ is the square root of the sum of squares of the coefficients of $T_n$ which implies that both the previous statements are equivalent to the uniform convergence of the coefficients to $0$ on $K$.

In particular, this establishes the case $k = 0$ of the lemma.  We will prove the lemma by induction on $k$ but first we must establish some basic properties of the coefficients of $\nabla^iT_n$.

For this purpose, assuming that $T$ is a $(p,q')$-tensor field, observe that the coefficients of $\nabla T$ are obtained from the equation
\begin{equation*}\label{tensorcovariantderivative}
\nabla T(Y,X_1,\ldots,X_{p+q'}) = \nabla_Y T(X_1,\ldots,X_{p+q'}) - \sum\limits_{i = 1}^{p+q'}T(X_1,\ldots, \nabla_YX_i,\ldots,X_{p+q'}) 
\end{equation*}
by substituting elements of the canonical basis for $Y,X_1,\ldots,X_{q'}$ and elements of the dual basis for $X_{q'+1},\ldots,X_{p+q'}$.

The first term above is simply the derivative in the direction of the basis vector $Y$ of a coefficient of $T$ while the other terms are products of the coefficients of $T$ with Christoffel symbols.

By induction one can establish that for each $i$ one has
\begin{enumerate}
 \item Each coefficient of $\nabla^{i}T$ is the sum of one $i$-th order partial derivative of a coefficient of $T$ with products of lower order partial derivatives coefficients of $T$ with partial derivatives of the Christoffel symbols of order less than or equal to $i$.
 \item Every partial derivative of order $i$ of each coefficient of $T$ appears in at least one of the aforementioned sums.
\end{enumerate}

Now assume that our lemma is true for $k-1$.

If $|\nabla^i T_n|_g$ converges to $0$ uniformly on $K$ for each $i \le k$ then by the induction hypothesis the partial derivatives of the coefficients of $T_n$ up to order $k-1$ converge uniformly to $0$ on $K$.  Using the properties of $\nabla^kT_n$ established above and the bound $\Gamma$ on partial derivatives of the Christoffel symbols it follows that the $k$-th order partial derivatives of the coefficients of $T_n$ converge to $0$ uniformly on $K$ as well.

Similarly if the partial derivatives up to order $k$ of the coefficients of $T_n$ converge to $0$ uniformly on $K$ then by the properties of $\nabla^iT_n$ established above and the bounds on the Christoffel factors one obtains that the coefficients of $\nabla^iT_n$ converge to $0$ uniformly for each $i \le k$ on $K$.   This implies (using the constant $C$ defined above) the claim on $|\nabla^iT_n|_g$.
\end{proof}

\section{Bounded geometry of leaves}\label{leafgeometrysection}

We now verify that the leaves of a compact foliation have uniformly bounded geometry.  This was used in the proof of Theorem \ref{foliationcompactnesstheorem}.

\begin{lemma}\label{leafgeometry}
If $X$ is a compact $d$-dimensional foliation then there exists $r \g 0$ and a sequence $\lbrace C_k: k \ge 0\rbrace$ such that all leaves belong to the space $\manifolds\left(d,r,\lbrace C_k \rbrace\right)$.
\end{lemma}
\begin{proof}
We have shown in Section \ref{closedness} that the norm of the $k$-th covariant derivative of the curvature tensor is a continuous function of the metric coefficients, the coefficients of the inverse matrix, and a finite number of their partial derivatives.  This implies (by looking at the leaf metrics in a foliated chart) that this norm is continuous on $X$ and hence has a global maximum $C_k$.

Let $h:\R^d \times T \to U \subset X$ be a foliated parametrization and for each $t \in T$ let $g_t$ be the Riemannian metric on $\R^d$ obtained by pullback under $x \mapsto h(x,t)$.

Let $\exp_{x,t}$ denote the exponential map of the metric $g_t$ at $x$ (i.e. $\exp_{x,t}:\R^d \to \R^d$ with $\exp_t(v) = \alpha(1)$ where $\alpha$ is the $g_{t}$-geodesic satisfying $\alpha(0) = x$ and $\alpha'(0) = v$).

By continuity of the solution to an ordinary differential equation with respect to the vector field (see \cite[Theorem B3, pg. 333]{duistermaat-kolk2000}) one has that $(x,t) \mapsto \exp_{x,t}$ is continuous when the codomain is endowed with the topology of $C^k$ convergence on compact subsets of $\R^d$.

In particular each $(x,t) \in \R^d \times T$ has a neighborhood $U$ on which there is a radius $r \g 0$ such that the operator norm of the difference between the differential of $\exp_{y,s}$ and the identity is less than $\frac{1}{2}$ at all points in $B_{2r}(0)$ for all $(y,s) \in U$.  This implies that $\exp_{y,s} - z$ is a contraction mapping $B_{2r}(0)$ into itself for all $z \in B_{r}(y)$.  In particular the injectivity radius of $g_s$ at $y$ is at least $r$.

It follows that each point $x \in X$ has a neighborhood on which there is a uniform positive lower bound for the leafwise injectivity radius at each point.  Covering $X$ by a finite number of these open sets one obtains that there is a global positive lower bound for the injectivity radius of all leaves.
\end{proof}

\section{Covering spaces and holonomy}\label{holonomysection}

In this section we recall some basic facts on Riemannian coverings and provide the definitions and results on holonomy which are relevant for Theorem \ref{semicontinuitytheorem}.

\subsection{Riemannian coverings}

By a Riemannian covering of a pointed complete connected Riemannian manifold $(M,o,g)$ we mean a pointed local isometry $f:N \to M$ from some pointed complete connected Riemannian manifold $(N,o_N,g_N)$ to $M$.  We sometimes omit the function $f$ and simply say that $N$ is a Riemannian covering of $M$ (meaning there exists at least one suitable $f$).

Any covering space $N$ in the sense of \cite[Chapter 1.3]{hatcher2002} can be made a Riemannian covering by constructing local charts which are compositions of the covering map and the local charts of the covered manifold $M$.  Reciprocally any Riemannian covering satisfies the `pile of disks' property for the preimage of balls of radius smaller than half the injectivity radius of $M$.

We recall that the fundamental group $\pi_1(M)$ of $(M,o,g)$ is the group of (endpoint fixing) homotopy classes of closed curves starting and ending at the basepoint $o$.  Any covering $f:N \to M$ induces a morphism $f_*$ from $\pi_1(N)$ to $\pi_1(M)$.

With these observations we restate \cite[Theorem 1.38]{hatcher2002} and the comment immediately after about ordering covering spaces as we shall use them.

\begin{lemma}[Classification of covering spaces]\label{coverings}
Let $(M,o,g)$ be a pointed complete Riemannian manifold.  For each subgroup $H$ of $\pi_1(M)$ there exists a Riemannian covering $f:N \to M$ with $f_*(\pi_1(N)) = H$ and this covering space is unique up to pointed isometries.  If two Riemannian coverings $N$ and $N'$ correspond to subgroups $H \subset H'$ then $N$ is a Riemannian covering of $N'$.
\end{lemma}

The Riemannian covering associated to the trivial subgroup of $\pi_1(M)$ is the universal covering which we denote by $\widetilde{M}$.

\subsection{Holonomy covering}

Let $X$ be a foliation and $h_i:U_i \to \R^d \times T_i$ (where $i = 1,2$) be foliated charts.

The charts $h_1,h_2$ are said to be compatible if there exists a homeomorphism $\psi: h_1(U_1 \cap U_2) \to h_2(U_1 \cap U_2)$ such that
\[h_2\circ h_1^{-1}(x,t) = (\varphi(x,t),\psi(t))\]
for a certain (automatically smooth with respect to $x$) function $\varphi$ and all $(x,t)$ with $h_1(x,t) \in U_1 \cap U_2$.  The map $\psi$ is called the holonomy from $h_1$ to $h_2$.

Notice that the Vinyl record foliation of Section \ref{vinylsection} can't be covered by only two compatible charts.

By a chain of compatible foliated charts we mean a finite sequence $h_i: U_i \to \R^d \times T_i$ indexed on $i=0,1,\ldots,r$ of foliated charts such that $U_i$ intersects $U_{i+1}$ and $h_i$ is compatible with $h_{i+1}$ for all $i=0,\ldots,r-1$.  The chain is closed if $h_r = h_0$.  The holonomy of the chain is the map $\psi_{r-1,r}\circ \cdots \circ \psi_{0,1}$ where $\psi_{i,i+1}$ is the holonomy from $h_i$ to $h_{i+1}$ and we assume the maximal possible domain for the composition.

A leafwise curve is a continuous function $\alpha:[0,1] \to X$ whose image is contained in a single leaf.  We say $\alpha$ is covered by a compatible chain of foliated charts $\lbrace h_i, i=0,\ldots,r\rbrace$ if there exists a finite sequence $t_0 = 0 \l \cdots \l t_r = 1$ such that $\alpha([t_i,t_{i+1}]) \subset U_i$ for all $i = 0,\ldots,r-1$ where $U_i$ is the domain of $h_i$.

A closed leafwise curve $\alpha:[0,1] \to X$ is said to have trivial holonomy if there exists a closed chain of compatible charts $\lbrace h_i: i=0,\ldots,r\rbrace$ covering $\alpha$ whose holonomy map is the identity on a neighborhood of $t \in T$ where $t$ is the second coordinate of $h_0(\alpha(0))$.

The holonomy covering $\widetilde{L_x}^{\holonomy}$ of a leaf $L_x$ is defined as the Riemannian covering corresponding (via Lemma \ref{coverings}) to the subgroup $H$ of homotopy classes of closed curves based at $x$ in $L_x$ which have trivial holonomy.

To show that this is well defined we must prove that:
\begin{enumerate}
 \item Each closed leafwise curve admits a covering by a compatible closed chain of foliated charts.
 \item The property of having trivial holonomy doesn't depend on the choice of covering.
 \item The property of having trivial holonomy is invariant under homotopy.
\end{enumerate}

A leaf $L_x$ is said to have trivial holonomy if $L_x$ is isometric to its holonomy cover (equivalently all leafwise closed curves based at $x$ have trivial holonomy).  The fact that this property doesn't depend on the basepoint $x$ follows from Lemma \ref{freehomotopy} below, this lemma also covers item 3 in the above list and shows that the holonomy cover is a normal covering space (although we will not use this fact).

\subsection{Trivial holonomy}

In this subsection we verify the claims necessary for the definition of the holonomy covering of a leaf.  We also provide a characterization of trivial holonomy which was used in the proof of Theorem \ref{semicontinuitytheorem}.

Recall that an atlas of a foliation is simply a collection of foliated charts whose domains cover the foliation.  We say one atlas refines another if the domain of each chart of the former is contained in the domain of some chart of the later.  We call an atlas consisting of pairwise compatible charts `admissible'.

\begin{lemma}\label{admissiblerefinement}
Every atlas of a compact foliation has an admissible refinement.
\end{lemma}
\begin{proof}
Let $A$ be an atlas of a compact foliation $X$.  Since $X$ is compact we may take a finite subatlas $B$ of $A$.  

Let $h:U \to \R^d \times T$ be a chart in $B$.  Given an open ball $D \subset \R^d$ and an open subset $S \subset T$ we may construct a chart $g:h^{-1}(D \times S) \to \R^d \times S$ by letting $g(x) = f(h(x))$ where $f$ acts as the identity on the second coordinate and a fixed diffeomorphism between $D$ and $\R^d$ on the second.  We call any such chart a restriction of $h$ and note that any two restrictions of the same chart are compatible.

Now let $h_i:U_i \to \R^d \times T_i$ be charts in $B$ for $i=1,2$.  Even if these charts aren't compatible the fact that they are foliated charts implies that each point $x$ in $U_1 \cap U_2$ has a neighborhood $V = h_1^{-1}(D \times S)$ where $D \subset \R^d$ is an Euclidean open ball and $S$ is an open subset of $T_1$, such that
\[h_2 \circ h_1^{-1}(y) = (\varphi(y,t),\psi(t))\]
on $h_1(V)$, where $\psi$ is a homeomorphism between certain open sets in $T_1$ and $T_2$.  This implies that restricting $h_1$ to $V$ one obtains a chart which is compatible with any restriction of $h_1$ or $h_2$.

Since there are only finitely many charts we may choose for each point $x$ a neighborhood, and a restriction of a certain chart in $B$ to this neighborhood which will be compatible with (the restrictions of) all charts in $B$.   The collection of such charts gives a compatible refinement $C$ of $A$.
\end{proof}

From the above result it follows that any closed curve has a covering by a compatible chain of foliated charts.

We now establish the fact that having trivial holonomy doesn't depend on the choice of covering.  We recall that the plaques of a foliated chart $h:U \to \R^d \times T$ are the sets of the form $h^{-1}(\R^d \times \lbrace t \rbrace)$.

\begin{lemma}
If $\alpha$ is a leafwise closed curve in a compact foliation $X$.  Then $\alpha$ has trivial holonomy if and only if for each sequence $\alpha_n$ of (possibly non-closed) leafwise curves which converge uniformly to $\alpha$ and any foliated chart $h:U \to \R^d \times T$ where $\alpha(0) \in U$ one has that $\alpha_n(0)$ and $\alpha_n(1)$ belong to the same plaque for $n$ large enough.
\end{lemma}
\begin{proof}
Observe that if $h_i:U_i \to \R^d \times T_i$ (where $i=1,2$) are compatible foliated charts and $\beta$ is a leafwise curve whose image is contained in $U_1 \cup U_2$ then there exists $t_i \in T_i$ ($i = 1,2$) such that $\beta$ is in the plaque $h_i^{-1}(\R^d \times \lbrace t_i\rbrace)$ whenever it is in $U_i$.  Furthermore $t_2 = \psi(t_1)$ where $\psi$ is the holonomy between the charts.

By definition $\alpha$ is covered by a closed chain of compatible charts $h_i:U_i \to \R^d \times T_i$ where $i = 0,\ldots,r$ and there exist $t_0 = 0 \l \cdots \l t_r = 1$ such that $\alpha([t_i,t_{i+1}]) \subset U_i$ for all $i=0,\ldots,r-1$.  Furthermore the holonomy $\psi$ of the chain is the identity on a neighborhood of the second coordinate of $h_0(\alpha(0))$.

Take $\epsilon \g 0$ such that
\begin{enumerate}
 \item Any leafwise curve $\beta:[0,1] \to X$ at distance less than $\epsilon$ from $\alpha$ one has $\beta([t_i,t_{i+1}]) \subset U_i$ for all $i=0,\ldots,r-1$.  
 \item Any point $p$ at distance less than $\epsilon$ from $\alpha(0)$ is of the form $h_0^{-1}(x,t)$ where $\psi(t) = t$.
\end{enumerate}

It follows that any leafwise curve at uniform distance less than $\epsilon$ from $\alpha$ will start and end in the same plaque.
\end{proof}

The first observation in the above proof yields the following.
\begin{corollary}\label{trivialholonomycorollary}
If $\alpha$ is a leafwise closed curve with trivial holonomy in a compact foliation $X$.  Then any closed chain of compatible charts $h_i:U_i \to \R^d \times T_i$ (where $i =0,\ldots,r$) which covers $\alpha$ has trivial holonomy in a neighborhood of the second coordinate of $h_0(\alpha(0))$.
\end{corollary}

By the leafwise distance between to points on the same leaf $L_x$ we mean the distance with respect to the Riemannian metric $g_{L_x}$.  

The following corollary amounts to the observation that if $x_n,y_n$ are two sequences of points converging to the same limit $x$ which belong to the same sequence of plaques with respect to some chart $h$ covering $x$, then the leafwise distance between $x_n$ and $y_n$ converges to $0$.
\begin{corollary}\label{nontrivialholonomy}
If a sequence of leafwise curves $\alpha_n:[0,1] \to X$ converges uniformly to a closed leafwise curve $\alpha$ and the leafwise distance between $\alpha_n(0)$ and $\alpha_n(1)$ doesn't converge to $0$, then $\alpha$ has non-trivial holonomy.
\end{corollary}

We say two leafwise curves $\alpha,\beta:[0,1] \to X$ are leafwise freely homotopic if they belong to the same leaf $L$ and there exists a continuous function $h:[0,1]\times [0,1] \to X$ such that $t \mapsto h(s,t) = h_s(t)$ is a leafwise closed curve in $L$ for all $s$, $h_0 = \alpha$ and $h_1 = \beta$.

\begin{lemma}\label{freehomotopy}
Let $X$ be a compact foliation and $\alpha:[0,1] \to X$ a closed leafwise curve with trivial holonomy.  Then any closed leafwise curve which is leafwise freely homotopic to $\alpha$ also has trivial holonomy.
\end{lemma}
\begin{proof}
Take an admissible finite atlas $A$ of $X$ and let $\epsilon \g 0$ be the Lebesgue number of the associated open covering.  It follows from Corollary \ref{trivialholonomycorollary} that if two closed curves belong to the same leaf and are at uniform distance less than $\epsilon$ and one of them has trivial holonomy then they both do.

Letting $h_s$ be a homotopy between $\alpha$ and $\beta$ one can find times $s_0 = 0 \l \ldots \l s_r = 1$ such that $h_{s_i}$ is at uniform distance less than $\epsilon$ from $h_{s_{i+1}}$ for all $i=0,\ldots,r-1$ from which the lemma follows.
\end{proof}

\section{Convergence of leafwise functions}\label{leafwiseconvergencesection}

In this section we justify the claims on convergence of immersions into foliations which were used in the proof of Theorem \ref{semicontinuitytheorem}.

\subsection{Adapted distances}

Let $X$ be a compact foliation.  Denote by $\leafdistance$ the leafwise distance in $X$ which is defined by
\[\leafdistance(x,y) = \left\lbrace\begin{array}{ll}
                                   d_{L_x}(x,y) &\text{ if }y \in L_x.
                                   \\ +\infty &\text{otherwise.}
                                   \end{array}\right.\]
where $d_{L_x}$ is the Riemannian distance on the leaf $L_x$.

We call a distance $d$ on $X$ adapted if it metricizes the topology of $X$ and satisfies $d(x,y) \le \leafdistance(x,y)$.

Consider a smooth Riemannian manifold $X$ foliated by smoothly immersed leaves each of which inherits the ambient Riemannian metric (e.g. any example in Section \ref{definitionsandexamplessection}).  The Riemannian distance between two points $x,y \in X$ is the infimum of the lengths of arbitrary curves connecting them while the leafwise distance is the infimum among leafwise curves.  Hence one clearly has that the Riemannian distance is adapted to the foliation.  This makes the following result plausible.

\begin{lemma}\label{adapteddistance}
Every compact foliation has an adapted distance.
\end{lemma}
\begin{proof}
Let $X$ be a compact foliation.  We will construct an adapted distance by averaging pseudodistances obtained by a local construction.

Let $h:V \to \R^d \times T$ be a foliated chart and let $g_t$ be the family of Riemannian metrics on $\R^d$ parametrized by $t \in T$ obtained by pushforward of the leafwise metrics under $h$.  Fix a complete distance $d_T$ on $T$ and metrizice $\R^d \times T$ by defining
\[\rho_1((x,t),(x',t')) = |x-x'| + d_T(t,t')\]
for all $(x,t),(x',t') \in \R^d\times T$.

We observe that because $X$ is compact any point in $\R^d \times T$ has a compact neighborhood, and it follows that $T$ is locally compact.

Fix $(x,t) \in \R^d \times T$ and consider a precompact open neighborhood $S \subset T$ of $t$.  The family of Riemannian metrics on the Euclidean ball $B_1(x)$ of the form $g_s$ for $s \in S$ is smoothly precompact.  Hence there exists a constant $C \g 0$ such that the $g_s$-length of any curve between points $y,y' \in \overline{B_1}(x)$ is at least $C|y-y'|$ for all $s \in S$.

For this constant $C$ we choose a continuous function $\varphi:\R^d \times T \to [0,C]$ which is strictly positive exactly on the set $B_1(x) \times S$ and zero outside of it and define for all $(y,s),(y',s') \in \R^d \times T$
\[\rho_2((y,s),(y',s')) = \inf\left\lbrace \sum\limits_{i =0}^{k-1} \frac{\varphi(y_i,s_i) +\varphi(y_{i+1},s_{i+1})}{2}\cdot\rho_1((y_i,s_i),(y_{i+1},s_{i+1}))\right\rbrace\]
where the infimum is among all $k \in \N$ and all finite chains in $\R^d \times T$ with $(y_0,s_0) = (y,s)$ and $(y_k,s_k) = (y',s')$.

Because one can reverse a chain and concatenate two of them one obtains that $\rho_2$ is symmetric and satisfies the triangle inequality.  Notice also that $\rho_2$ is zero for any pair of points not in $B_1(x)\times S$.

Now consider $(y,s) \in B_1(x) \times S$ and the function $f(y',s') = \rho_2((y,s),(y',s'))$.  By the triangle inequality $f$ is constant outside of $B_1(x)\times S$.  Given $(y',s') \neq (y,s)$ one may choose $r \g 0$ such that the $\rho_1$-ball $B_{\rho_1,r}(y,s)$ of radius $r$ centered at $(y,s)$ doesn't contain $(y',s')$ and the values of $\varphi$ on this ball are bounded from below by a positive constant $\epsilon$.  For any finite chain $(y_i,s_i)$ joining $(y,s)$ and $(y',s')$ one may take the first $k$ such that $(y_k,s_k) \notin B_{\rho_1,r}(y,s)$ and since $\rho_1$ is a distance one has
\[\sum\limits_{i =0}^{k-1} \frac{\varphi(y_i,s_i) +\varphi(y_{i+1},s_{i+1})}{2}\cdot\rho_1((y_i,s_i),(y_{i+1},s_{i+1})) \ge \frac{1}{2}\epsilon r.\]

Hence $f$ is zero only at $(y,s)$.  Combined with the inequality $\rho_2 \le C \rho_1$ one obtains that $\rho_2$ is a continuous bounded pseudodistance on $\R^d \times T$ which is an actual distance when restricted to $B_1(x) \times S$ and which is zero on pairs of points not belonging to $B_1(x) \times S$.

Hence the pullback of $\rho_2$ to $V$ via $h$ can be extended to a bounded continuous pseudodistance $\rho:X \times X \to [0,+\infty)$ which is an actual distance when restricted to the open set $U = h^{-1}(B_1(x) \times S)$ and which is zero on pairs of points not belonging to this set.

We will now establish that $\rho(p,p') \le d_L(p,p')$ whenever $p$ and $p'$ are on the same leaf.  The only interesting case (i.e. $\rho \neq 0$) is if either $p$ or $p'$ belong to $U$.  Suppose $p \in U$ and let $\alpha:[0,1] \to X$ be the leafwise geodesic of length $\leafdistance(p,p')$ joining $p$ and $p'$.  There are two cases to consider: either $\alpha$ leaves $U$ or it doesn't.

If $\alpha([0,1]) \subset U$ then taking $\beta = h\circ \alpha$ and setting $(y,s) = \beta(0)$ and $(y',s') = \beta(1)$ one obtains that $s = s'$.  Since $\rho(p,p') = \rho_2((y,s),(y',s)) \le C|y-y'|$ which is a lower bound for the $g_s$ length of any curve joining $y$ and $y'$ one obtains that $\rho(p,p') \le \leafdistance(p,p')$ as claimed.

Now suppose that $\alpha$ leaves $U$ and take $T \in [0,1)$ so that $\beta = h\circ \alpha$ is well defined on $[0,T]$ and $\beta(T) \notin B_1(x)\times S$.   Setting $(y,s) = \beta(0)$ and $(y',s') = \beta(T)$ one obtains once again that $s = s'$ and that the $g_s$-length of $\beta$ is at least $C|y-y'|$ which is larger than $\rho(\alpha(0),\alpha(T)) = \rho_2((y,s),(y',s'))$.  If $p' \notin U$ we are done since $\rho(\alpha(T),p') = 0$.  Otherwise we take $T \l T_2\l 1$ so that $\beta_2 = h\circ\alpha$ is well defined on $[T_2,1]$ and repeat the preceeding argument to obtain that $\rho(p,p') \le \rho(p,\alpha(T)) + \rho(\alpha(T_2),p')$ is less than the length of $\alpha$ as claimed.

We have succeeded in constructing for each $p \in X$ an open neighborhood $U$ and a continuous bounded pseudodistance $\rho$ which is an actual distance when restricted to $U\times U$ and which satisfies $\rho(q,q') \le \leafdistance(q,q')$.   Covering $X$ with a finite number of such neighborhoods $U_1,\ldots,U_n$ with associated pseudodistances $\rho_1,\ldots,\rho_n$ and setting $d(p,q) = \frac{1}{n}\sum\limits_{i = 1}^{n} \rho_i(p,q)$ one obtains an adapted distance for the foliation $X$. 
\end{proof}

\subsection{Convergence of leafwise immersions}

We conclude the section with the following result which was used to proved Theorem \ref{semicontinuitytheorem} (recall that a function into a foliation is said to be leafwise if its image is contained in a single leaf).

\begin{lemma}\label{leafwiseconvergence}
Let $X$ be a compact foliation and $(M,o,g)$ be a complete pointed Riemannian manifold.  If $f_n:B_r(o) \to X$ is a sequence of leafwise functions such that $|f_n^*g_{L_{f_n(o)}} - g|_g$ converges to $0$ uniformly then there exists a subsequence converging locally uniformly to a leafwise local isometry $f:B_r(o) \to X$.
\end{lemma}
\begin{proof}
The hypothesis implies that the $f_n$ are locally uniformly Lipschitz with respect to any adapted distance on $X$.  By the Arzelà-Ascoli Theorem there exists a subsequence $f_{n_k}$ which converges locally uniformly to a limit $f$.

Given $x \in B_r(o)$ we may consider a foliated parametrization $h:\R^d \times T \to U \subset X$ of a neighborhood of $f(x)$ such that $f(x) = h(0,t)$.  For each $s \in T$ let $g_s$ be the Riemannian metric on $\R^d$ obtained by pullback under $z \mapsto h(z,s)$.

Let $\epsilon \g 0$ be such that the $g_t$-ball centered at $0$ of radius $2 \epsilon$ is bounded with respect to the Euclidean metric on $\R^d$, the ball of radius $2\epsilon$ centered at $x$ is contained in $B_r(o)$, and $f(B_{2\epsilon}(x))$ is contained in the open set parametrized by $h$.

We will show that $\pi_1 \circ h^{-1}\circ f$ is an isometry from $B_\epsilon(x)$ into $\R^d$ with the Riemannian metric $g_t$ where $\pi_1:\R^d \times T \to \R^d$ is the projection onto the first coordinate.

For this purpose consider $y \in B_\epsilon(x)$.  The sequences $p_n = \pi_1\circ h\circ f_n(x)$ and $q_n = \pi_1\circ h\circ f_n(y)$ are eventually well defined and converge to $p = \pi_1\circ h\circ f(x)$ and $q = \pi_1\circ h\circ f(y)$ respectively.  Furthermore letting $t_n$ be the common coordinate in $T$ of $h\circ f_n(x)$ and $h\circ f_n(y)$ one has that the $g_{t_n}$-distance between $p_n$ and $q_n$ converges to $d_M(x,y)$ ($d_M$ being the Riemannian distance on $M$)\begin{footnote}{Here we use the fact that, if $n$ is large enough, the $g_{t_n}$-ball of radius $\epsilon$ is bounded in $\R^d$.}\end{footnote}.   Since $g_{t_n}$ converges smoothly on compact sets to $g_t$ one has that the $g_t$-distance between $p$ and $q$ equals $d_M(x,y)$ as claimed.
\end{proof}

\end{document}